\documentclass[notitlepage]{scrartcl}
\usepackage[shortlabels]{enumitem}
\usepackage{amsmath}
\usepackage{amssymb}
\usepackage{amsthm}
\usepackage{abstract}
\usepackage{xcolor}
\usepackage{comment}
\usepackage{mathtools}
\usepackage{graphicx}
\usepackage{caption}
\usepackage{subcaption}
\usepackage{float}
\usepackage{hyperref}
\usepackage{cleveref}
\usepackage{tikz}
\hypersetup{
    colorlinks=true,
    linkcolor=blue,
    filecolor=magenta,      
    urlcolor=cyan
    }
\makeatletter
\newcommand*{\rom}[1]{\expandafter\@slowromancap\romannumeral #1@}
\newcommand*{\barfix}[2][.175ex]{%
  \mathpalette{\@barfix{#1}}{#2}%
}
\newcommand*{\@barfix}[3]{%
  \vbox{%
    \kern#1\relax
    \hbox{$#2#3\m@th$}%
  }%
}
\makeatother

\newtheorem{theorem}{Theorem}
\newtheorem{thm}{Theorem}[section]
\newtheorem{corollary}[thm]{Corollary}
\newtheorem{lemma}[thm]{Lemma}

\newtheorem{claim}[thm]{Claim}

\newtheorem{conjecture}[thm]{Conjecture}

\usepackage[margin=1in]{geometry} 

\providecommand\given{\nonscript\:\ifthenelse{\equal{\delimsize}{}}{\big\vert}{\delimsize\vert}\nonscript\:\mathopen{}}
\let\Pr\undefined
\DeclarePairedDelimiterXPP\Pr[1]{\mathbb{P}}(){}{#1}
\DeclarePairedDelimiterXPP\Ex[1]{\mathbb{E}}{[}{]}{}{#1}

\title{\vspace{-1.5cm}Supercritical Site Percolation on Regular Graphs}
\author{Sahar Diskin \footnote{Department of Mathematics, ETH Z\"urich. Email: \href{mailto:Sahardiskinmail@gmail.com}{Sahardiskinmail@gmail.com}.} \and 
Michael Krivelevich \footnote{School of Mathematical Sciences, Tel Aviv University, Tel Aviv 6997801, Israel. \newline Email:
\href{mailto:krivelev@tauex.tau.ac.il}{krivelev@tauex.tau.ac.il}. Research supported in part by NSF-BSF grant 2023688.} \and 
Itay Markbreit \footnote{School of Mathematical Sciences, Tel Aviv University, Tel Aviv 6997801, Israel. \newline Email:
\href{mailto:markbreit@mail.tau.ac.il}{markbreit@mail.tau.ac.il}.}}

\begin{document}
\date{}
\maketitle

\begin{abstract}
    We consider site (vertex) percolation on $d$-regular graphs, for both constant-degree and growing-degree cases. We give sufficient, and relatively tight, conditions for the emergence of the ``Erd\H os-R\'enyi component phenomenon" in the supercritical regime $p=\frac{1+\epsilon}{d-1}$: namely, the appearance of a unique giant component of order $n/d$ in the percolated subgraph, with all other components being of size $O(\log n)$. Our main results apply both to the $d$-dimensional hypercube and to pseudo-random graphs, and resolve two open questions in these cases. We further discuss differences (and similarities) between bond (edge) percolation setting and site percolation setting.
\end{abstract}

\section{Introduction and Main Results}\label{intro}

Given a host graph $G=(V,E)$ and probability $p\in [0,1]$, in \textit{$p$-site-percolation} we form a random subset $V_p \subseteq V$ by including each vertex $v \in V$ in $V_p$ independently and with probability $p$, and consider the subgraph $G[V_p]$.
A tightly related model, which has perhaps been more studied between the two, is the \textit{$p$-bond-percolation}. There, one forms a random subgraph $G_p$ of a host graph $G=(V,E)$ by retaining each edge of $G$ with probability $p$. When the probability $p$ is clear from the context, we omit it and write site percolation and bond percolation.

In 1957, Broadbent and Hammersley \cite{zbMATH03148802} initiated the study of percolation theory in order to model the flow of fluid through a medium with randomly blocked channels. Since then, the theory of percolation has been studied extensively (see \cite{zbMATH05076923}, \cite{zbMATH01301486} and \cite{zbMATH03826957} for systematic coverage).

We first review some results in the bond percolation setting. The classical binomial random graph $G_{n,p}$ is an instance of bond percolation (indeed, $(K_n)_p\cong G_{n,p}$). A classical result of Erd\H os and R\'enyi from 1960 \cite{zbMATH03168330} states that $G_{n,p}$ undergoes a fundamental phase transition with respect to its component structure when the expected average degree is around $1$ (that is, $n\cdot p \approx 1$).

Formally, for an integer $d \geq 3$ and $p \in (0, 1)$, let $q \coloneqq q(p)$ be the unique solution in $(0, 1)$ of the equation:
\begin{align*}
    q=(1-p+pq)^{d-1}.
\end{align*}
This is the extinction probability of a Galton-Watson tree with offspring distribution $Bin(d - 1, p)$ (see, for example, \cite{zbMATH03410334} and \cite[Section 4.3.4]{zbMATH06976511}). Consider the probability $y \coloneqq y(p)$ that the root of an infinite $d$-regular tree
belongs to an infinite cluster after $p$-bond-percolation. Then, $y$ is given by

\begin{equation*}
\begin{split}
    y&=\sum_{i=1}^d \binom{d}{i}p^i(1-p)^{d-i}(1-q^i)=
    \sum_{i=1}^d \binom{d}{i}p^i(1-p)^{d-i}-\sum_{i=1}^d \binom{d}{i}(pq)^i(1-p)^{d-i}\\
    &=1-(1-p+pq)^d=1-q(1-p)-pq^2.
\end{split}
\end{equation*}
Let us note that when $d$ tends to infinity and $p=\frac{1+\epsilon}{d-1}$, for small constant $\epsilon>0$, then $y\coloneqq y(\epsilon)$ is asymptotically equal to the unique solution in $(0,1)$ of $1-y=\exp(-(1+\epsilon)y)$.
We can now explain the result of Erd\H os and R\'enyi. They showed that when $p=\frac{1-\epsilon}{n}$, typically all components of $G_{n,p}$ are of size $O(\log n/\epsilon^2)$, and when $p=\frac{1+\epsilon}{n}$, there typically exists a unique giant component $L_1$ in $G_{n,p}$ of order $(1 + o(1))y(\epsilon)n$, and all other components of $G_{n,p}$ are of size $O(\log n/\epsilon^2)$ (which is exactly the same order as in the subcritical regime).

The phenomenon Erd\H os and R\'enyi discovered for bond percolation on the complete graph $K_n$ was extended to other (families of) graphs, such as bond percolation on pseudo-random $(n,d,\lambda)$-graphs \cite{zbMATH02037533} and on the hypercube $Q^d$ \cite{zbMATH03769676,zbMATH00035138} (see also \cite{arXiv:2311.07210}). Recently, Diskin and Krivelevich \cite{zbMATH08098276,diskin2025componentslargesmalli} found that assuming optimal vertex and edge expansions up to sizes $O(\log n)$ and $O(d\log n)$, respectively, together with a very mild `global' edge expansion, are sufficient to typically have a unique giant component in the percolated subgraph of order $y(\epsilon)n$, and all other components are of typical order of $O(\log n/\epsilon^2)$, where $p=\frac{1+\epsilon}{d-1}$. They also showed that these conditions are tight in a certain well-defined quantitative sense.

In this paper, we consider the $p$-site-percolation setting and aim to extend the bond percolation results to the site percolation setting. In the site percolation setting, the anticipated result is slightly different from the aforementioned bond percolation result. We now formulate the phenomenon we aim to establish.
First, we observe that the probability that the root of an infinite $d$-regular tree is in the infinite cluster after site percolation is the same, up to the sampling of the root itself, as the probability that the root is in an infinite cluster after bond percolation.
Therefore, if we define $x:=x(p)$ to be the probability that the root of an infinite $d$-regular tree belongs to an infinite cluster after $p$-site-percolation, then

\begin{equation} \label{eq}
    x=p\cdot y=p-pq(1-p)-p^2q^2.
\end{equation}

It can be seen that when $d$ tends to infinity and $p=\frac{1+\epsilon}{d-1}$, for small constant $\epsilon>0$, then the asymptotic behavior of $x$ is $x=\frac{2\epsilon-O\left(\epsilon^2\right)}{d}$.

We say that a $d$-regular host graph $G$ on $n$ vertices exhibits the Erd\H os-R\'{e}nyi component phenomenon (ERCP), if given a sufficiently small constant $\epsilon>0$ and $p=\frac{1+\epsilon}{d-1}$, $G[V_p]$ possesses \textbf{whp}\footnote{With high probability, that is, with probability tending to one as $n$ tends to infinity.} the following properties: a unique giant component of asymptotic order $x\cdot n$ emerges, and all other components have order $O\left(\frac{\log (n/d)}{\epsilon^2}\right)$. Note that we focus on the supercritical regime in this paper since the subcritical regime, where $p=\frac{1-\epsilon}{d-1}$, is completely solved in this context. Indeed, Diskin and Krivelevich \cite{zbMATH07751060} showed that in this regime \textbf{whp} all components in $G[V_p]$ are of size $O\left(\frac{\log (n/d)}{\epsilon^2}\right)$ (see also \cite{zbMATH06534995}).

Let us mention a few prominent cases where ERCP has been partially established: for the $d$-dimensional Hamming torus \cite{zbMATH06285647}, for pseudo-random $(n,d,\lambda)$-graphs \cite{zbMATH07751060,zbMATH06534995} and for the binary hypercube $Q^d$ \cite{zbMATH00736289} it was shown that \textbf{whp} there is a unique giant component (in the percolated subgraph) of asymptotic order $x\cdot n$. For the binary hypercube $Q^d$, Bollob\'as, Kohayakawa, and \L uczak \cite{zbMATH00736289} further showed that all other components of $Q^d[V_p]$ are of size at most $d^{10}$. This result has been improved in \cite{zbMATH07770097}, showing that all other components of $Q^d[V_p]$ are of typical size of $O\left(d/\epsilon^5\right)$, thus nearly establishing ERCP in full.

The purpose of this paper is to establish sufficient, and relatively tight, conditions under which ERCP will hold in full. Note that in fact these conditions will capture the aforementioned cases (and thereby resolve open problems in their context). We show that a very mild `global' edge expansion, together with some control on the ‘local’ expansion of the graph, provide sufficient conditions for a $d$-regular graph $G$ to exhibit ERCP (Theorem~\ref{general} and Theorem~\ref{constant}). In addition, in the case where $d\coloneqq d(n)$ tends to infinity, we show in Theorem~\ref{no_second} that the conditions of Theorem~\ref{general} are relatively tight. Interestingly, achieving the same phenomenon in site- and bond-percolation necessitates different `local' assumptions. We also give a concrete statement handling $(n,d,\lambda)$-graphs in full (Theorem~\ref{n,d,lambda}). Before stating our results, we note that in the growing-degree case, we write $p=\frac{1+\epsilon}{d}$ instead of $p=\frac{1+\epsilon}{d-1}$, which we do for the clarity of representation.

Our first main result provides sufficient conditions for the full emergence of ERCP for growing-degree graphs.

\begin{theorem} \label{general}
    Let $c_1,c_2>0$ and let $0<\alpha\le 1$ be constants. Let $\epsilon>0$ be a sufficiently small constant. Then, there exist $c_3\coloneqq (c_2,\epsilon,\alpha)>0$ and $C\coloneqq C(\alpha)>0$ constants such that the following holds. Let $n$ be a sufficiently large integer, and let $d$ be an integer such that $d=\omega_n(1)$.
    
    Let $G=(V,E)$ be a $d$-regular graph on $n$ vertices, satisfying:
    \begin{enumerate}[(P\arabic*{})]
        \item For every $U\subseteq V(G)$ with $|U|\leq \frac{n}{2}$, $e_G(U,U^c)\ge c_1|U|$; \label{prop0A}
        \item For every $U\subseteq V(G)$ with $|U|\leq c_3d^{1-\alpha}\log n$, $|N_G(U)|\ge c_2d^\alpha|U|$; \label{prop1A}
        \item For every $U\subseteq V(G)$ with $|U|\leq 4d^{\frac{4}{\alpha}+2}\log^C n$, $e_G(U,U^c)\ge (1-\frac{\epsilon^2}{1000})d|U|$. \label{prop2A}
    \end{enumerate}
    Set $p=\frac{1+\epsilon}{d}$. Then, \textbf{whp} $G[V_p]$ contains a unique giant component $L_1$, satisfying $|V(L_1)| = (1 + o(1))x\cdot n$, where $x$ is defined according to \eqref{eq}, and all other components are of size at most $\frac{100}{\epsilon^2}\log n$.
\end{theorem}

Few comments are in place. The power of $d$ in assumption \ref{prop2A} is somewhat arbitrary and is a technical artifact of the proof. This condition implicitly impose an additional assumption on $d$, which is $d<n^{\gamma}$, for some constant $\gamma\coloneqq \gamma(\alpha)>0$. We note that this assumption is not very far from the one that must be taken into account --- as in the site percolation setting, the typical size of $V_p$ is $n/d$, the assumption $d=o(n)$ is necessary for us to make any \textbf{whp} statement to begin with. In addition, under this assumption on $d$, we remark that $\log n$ as written at the end of Theorem~\ref{general} is equivalent to $\log (n/d)$ as in the definition of ERCP.

Nevertheless, we will soon see (in Theorem~\ref{no_second}) that the power of $d$ in assumption \ref{prop2A} cannot be lowered below $2$. This stands in contrast to the aforementioned bond percolation setting \cite{diskin2025componentslargesmalli}, where, even when assuming optimal vertex and edge expansions up to size $O(\log n)$ and $O(d\log n)$, respectively (that is, taking $0$ and $1$ as the power of $d$ in assumptions \ref{prop1A} and \ref{prop2A}, respectively, and $C=1$), the full bond-percolation analogue of ERCP is satisfied.

The family of graphs satisfying the properties of Theorem~\ref{general} is quite wide, and includes random $d$-regular graphs, certain families of expanders and the $d$-dimensional hypercube $Q^{d}$. Below we discuss several concrete examples, and also remark on open problems the above theorem resolves.

\paragraph{}
\textbf{Hypercube.} The $d$-dimensional hypercube $Q^d$ is the graph with the vertex set $V(Q^d)=\{0,1\}^d$, where two vertices are adjacent if they differ in exactly one coordinate. The classical isoperimetric results of Harper \cite{zbMATH03353132,zbMATH03254683} show that the hypercube $Q^d$ satisfies \ref{prop0A}-\ref{prop2A}. First, we have that $n\coloneqq \left|V\left(Q^d\right)\right|=2^d$, and thus $d=\log_2 n$. By \cite{zbMATH03353132}, for every $U\subseteq V\left(Q^d\right)$, $e_G(U,U^c)\ge |U|(d-\log_2 |U|)$, and hence assumptions \ref{prop0A} and \ref{prop2A} are satisfied in $Q^d$. Finally, by \cite{zbMATH03254683}, for a set $U$ of size polynomial in $d$, the vertex expansion of $U$ is of order $\Omega(d)$, and thus \ref{prop1A} holds (even with $\alpha=1$). Therefore, Theorem~\ref{general} shows that $Q^d$ exhibits the full ERCP.

As mentioned in the introduction, Bollob\'as, Kohayakawa, and \L uczak \cite{zbMATH00736289} showed that, \textbf{whp}, there is a unique giant component in $Q^d[V_p]$, whose asymptotic size is $x\cdot n$, and conjectured that the other components besides the giant component are of asymptotic size at most $\gamma d$, where $\gamma$ is a constant depending only on $\epsilon$. Diskin and Krivelevich \cite[Theorem 1.2]{zbMATH07770097} proved this conjecture, with $\gamma$ of order $\frac{1}{\epsilon^5}$.
Theorem~\ref{general} recovers the aforementioned results, and in fact, resolves the question in full, obtaining the optimal dependence on $\epsilon$, with $\gamma$ of order $\frac{1}{\epsilon^2}$.

\paragraph{}
\textbf{Pseudo-Random Graphs.} As mentioned above, Theorem~\ref{general} can be applied to certain families of expanders/pseudo-random graphs. Let us focus on the family of $(n, d, \lambda)$-graphs, which is defined as follows. Assume that the eigenvalues of (the adjacency matrix of) a graph $G$ are ordered in the non-increasing order $\lambda_1 \geq \dots \geq \lambda_n$. The largest eigenvalue of any $d$-regular graph is easily seen to be $d$, sometimes referred to as the trivial eigenvalue of $G$. A graph $G$ is an $(n, d, \lambda)$-\textit{graph} if $G$ is a $d$-regular graph on $n$ vertices and all the eigenvalues of $G$, but the first one, are at most $\lambda$ in their absolute values. We refer the reader to the survey \cite{zbMATH05037081} for an extensive discussion on $(n,d,\lambda)$-graphs and other notions of pseudo-random graphs. Theorem~\ref{general} holds for $(n, d, \lambda)$-graphs with $\lambda\le d^{1-\gamma}$, where $\gamma>0$ is a constant. Indeed, by Alon and Milman \cite{zbMATH03875335} (see also Claim \ref{eml}), \ref{prop0A} and \ref{prop2A} hold for these graphs; standard results on the vertex expansion of $(n,d,\lambda)$-graphs show that \ref{prop1A} holds since we assumed that $\lambda\le d^{1-\gamma}$ (see, e.g., Lemma A.2 in \cite{arXiv:2503.06826}). In fact, our proof method is rather robust, and a slight change in the parametrization allows us to extend Theorem~\ref{general} to all $(n,d,\lambda)$-graphs with $\frac{\lambda}{d}<\delta$, for some constant $\delta>0$:

\begin{theorem} \label{n,d,lambda}
    For every small enough $\epsilon>0$, there exists $\delta>0$ such that the following holds. Let $n$ be sufficiently large and let $G=(V,E)$ be an $(n,d,\lambda)$-graph, where $d=\omega_n(1)$, $d=o(n)$, and $\frac{\lambda}{d}<\delta$. Set $p=\frac{1+\epsilon}{d}$. Then, \textbf{whp} $G[V_p]$ contains a unique giant component $L_1$, satisfying $|V(L_1)| = (1 + o(1))x\cdot n$, where $x$ is defined according to \eqref{eq}, and all other components are of size at most $\frac{100}{\epsilon^2}\log \left(\frac{n}{d}\right)$.
\end{theorem}

Let us note that Theorem~\ref{n,d,lambda} recovers and significantly strengthens the result of \cite{zbMATH07751060}, by showing not only that typically the giant component has the correct order, but also the second-largest component, thus resolving the problem for $(n,d,\lambda)$-graphs in full.

We now turn to discuss the tightness of Theorem~\ref{general}. First, notice that a global assumption on the graph like \ref{prop0A} is necessary, as for example, a disjoint union of two random $d$-regular graphs \textbf{whp} fails to contain a unique giant component after site percolation. Regarding \ref{prop1A} and \ref{prop2A}, we first recall some related results in the bond percolation setting. In \cite{zbMATH08098276,diskin2025componentslargesmalli}, it was established that \ref{prop0A} together with nearly optimal vertex and edge expansions up to sizes $O(\log n)$ and $O(d\log n)$, respectively, are sufficient to guarantee the bond-percolation analogous ERCP.  Conversely, if the `local' expansion conditions are omitted, the bond-percolation analogous ERCP no longer holds \cite[Theorem 3]{diskin2024percolationisoperimetry}. The following result shows that, in the site percolation setting, stronger `local' expansion conditions are necessary in order to ensure the ERCP. This reveals a distinction between the site- and bond-percolation settings, while also demonstrating that Theorem~\ref{general} is relatively tight. 

\begin{theorem} \label{no_second}
     For every constant $\alpha>0$, for every sufficiently small constant $\epsilon>0$ and for every $0<b\le \frac{1}{500\epsilon}$, there exists $c\coloneqq c(\alpha)>0$ which is sufficiently small in terms of $\alpha$, and there are infinitely many pairs $(d,n)\in\mathbb{N}^2$, with $d=\omega_n(1)$ and $d<n^{1/5}$, for which there exists a $d$-regular graph $G=(V,E)$ on $n$ vertices such that:
    \begin{enumerate}[(Q\arabic*{})]
        \item For every $U\subseteq V(G)$ with $|U|\leq \frac{n}{2}$, $e_G(U,U^c)\geq b|U|$; \label{prop0B}
        \item For every $U\subseteq V(G)$ with $|U|\leq c\cdot d\log n$, $|N_G(U)|\ge \left(1-\alpha\right)d|U|$; \label{prop1B}
        \item For every $U\subseteq V(G)$ with $|U|\leq \frac{\alpha}{2}\cdot d^2\log n$, $e_G(U,U^c)\ge \left(1-\alpha\right)d|U|$, \label{prop2B}
    \end{enumerate}
    and, moreover, $G$ satisfies the following. Set $p=\frac{1+\epsilon}{d}$. \textbf{Whp}, $G[V_p]$ contains at least two components of size at least $\epsilon d\log n$.
\end{theorem}

We note that, in fact, it can further be shown that for the graph $G$ constructed in the proof of Theorem~\ref{no_second}, \textbf{whp} $G[V_p]$ contains a giant component of size at least $(1-o(1))x\cdot n$.

A gap arises between the exponent of $d$ in the assumptions stated in Theorem~\ref{general} (with $\alpha=1$ there) and in Theorem~\ref{no_second}. Note that using concepts from the proof of Theorem 1 in \cite{diskin2025componentslargesmalli}, one can narrow the gap between the exponents of $d$ in the third assumption of the two theorems, but still not achieve $2$ in the exponent. A natural open question therefore, which is discussed in more detail in Section \ref{discussion}, is to determine the correct exponent, which would be both sufficient and necessary for ERCP.

Finally, we address graphs with constant (fixed) degree. As in the case where $d=\omega(1)$, our goal here is to find sufficient properties of a graph to have the full ERCP. The differences and similarities between the two regimes of $d$ can be seen in more detail in Section \ref{Pconstant}. Our result is strongly inspired by the analogous result in the bond percolation setting appearing in \cite{zbMATH08098276}. The conditions in the two settings (bond- and site-percolation) are the same, and the proofs are quite similar. However, in the site percolation setting, the second condition \ref{prop2C} plays an additional role, not utilized in the bond percolation setting, as it relates to the local vertex expansion of the graph. Before stating the result, we emphasize that when $d=O(1)$ we write $p=\frac{1+\epsilon}{d-1}$ and cannot replace it by $p=\frac{1+\epsilon}{d}$ as in the growing-degree case.

\begin{theorem}\label{constant}
    Let $n$ be a sufficiently large integer, and let $3\le d=O(1)$ be an integer. Let $1<\alpha<d-1$ and let $b>0$ be constants. Then, there exist constants $\delta:=\delta(\alpha)>0$ and $C:=C(\alpha)>0$ such that the following holds.
    Let $G=(V,E)$ be a $d$-regular graph on $n$ vertices, satisfying:
    \begin{enumerate}[(R\arabic*{})]
        \item For every $U\subseteq V(G)$ with $|U|\leq \frac{n}{2}$, $e(U,U^c)\geq b|U|$; \label{prop0C}
        \item For every $U\subseteq V(G)$ with $|U|\leq \log^C n$, $e(U)\le \left(1+\delta\right)|U|$. \label{prop2C}
    \end{enumerate}
    Set $p=\frac{\alpha}{d-1}$. Then, \textbf{whp} $G[V_p]$ contains a unique giant component $L_1$, satisfying $|V(L_1)| = (1 + o_\delta(1))x\cdot n$, where $x$ is defined according to \eqref{eq}, and all other components are of size at most $\frac{9\alpha}{\left(\frac{1-\delta}{1+\delta}\alpha-1\right)^2}\log n$.
\end{theorem}

Let us present some applications of this result. Observe that any $n$-vertex graph $G$ in which every two cycles of length at most $C\log\log n$ are at distance at least $C\log\log n$ satisfies \ref{prop2C}. Hence, $d$-regular $n$-vertex expanders with girth $\Omega(\log\log n)$ have the behavior in the supercritical regime postulated by Theorem~\ref{constant}. Furthermore, since a random $d$-regular graph typically satisfies this (see, for example, \cite{zbMATH01342092}), and also \textbf{whp} satisfies \ref{prop0C} (see \cite{zbMATH04101270}), we can apply our result for random $d$-regular graphs. Janson \cite{zbMATH05636576} showed the typical emergence of a giant component of size $\Theta(n)$ and all other components of size $o(n)$ in the percolated subgraph of a random regular graph. Our theorem recovers this result and moreover improves the upper bound on the typical size of the second-largest component.

The paper is structured as follows. In Section \ref{pre} we set our notation, describe a modification of the Breadth First Search (BFS) algorithm, and collect several useful lemmas. Section \ref{Pgeneral} is dedicated to the proof of Theorem~\ref{general}. In Sections \ref{Pn,d,lambda} and \ref{Pconstant} we show how to adapt the proof of Theorem~\ref{general} to $(n,d,\lambda)$-graphs (Theorem~\ref{n,d,lambda}), and to the case where $d=O(1)$ (Theorem~\ref{constant}), respectively. In Section \ref{Pno_second} we present a construction proving Theorem~\ref{no_second}. Finally, we discuss our results and avenues for future research in Section \ref{discussion}.

\section{Preliminaries} \label{pre}

Let us first state some notational conventions used in the paper. Consider a graph $H$ with $V\coloneqq V(H)$ and a vertex $v \in V$. Let $d_H(v)$ be the degree of $v$ in the graph $H$. Let $C_H(v)$ be the connected component in $H$ that contains $v$. When we perform site percolation on the graph $H$, we say that $C_{H[V_p]}(v)=\emptyset$ if $v$ falls outside of $V_p$. For an integer $r$, let $B_H(v, r)$ be the ball of radius $r$ in $H$ centered at $v$, that is, the set of all vertices of $H$ at distance of at most $r$ (edges) from $v$. For $u, v \in V (H)$ which lie in the same connected component, denote the distance in $H$ between $u$ and $v$ by $dist_H(u, v)$. For $S \subseteq V (H)$, set $dist_H(u, S) \coloneqq \min_{v\in S} dist_H(u, v)$ if there is a path from $S$ to $u$ in $H$, and set $dist_H(u, S)=\infty$ otherwise. As is fairly standard, $E_H(S, S^c)$ represents the set of edges in $H$ with one endpoint in $S$ and the other in $S^c\coloneqq V(H) \backslash S$, and $E_H(S)$ denotes the set of edges in the induced subgraph $H[S]$. We use $e_H(S, S^c) \coloneqq |E_H(S, S^c)|$ and $e_H(S) \coloneqq |E_H(S)|$. We denote by $N_H(S)$ the external neighborhood of $S$ in $H$, and denote by $N_H[S]$ the closed neighborhood of $S$, that is, $N_H[S]=S\cup N_H(S)$. When the base graph $H$ is clear from the context, the subscript may be omitted. All logarithms are with the natural base. Throughout the paper, we systematically ignore rounding signs as long as this does not affect the validity of our arguments.

\subsection{A modified Breadth First Search process}

We will utilize the following modification (whose definition is similar to the one in \cite{zbMATH08098276,zbMATH06534995}) of the classical Breadth First Search (BFS) exploration algorithm. The algorithm receives as input a graph $G = (V, E)$ with an ordering $\sigma$ on its vertices, and a sequence $(X_i)_{i=1}^{|V|}$ of i.i.d. Bernoulli$(p)$ random variables.

We maintain four sets throughout the process: $S$, the set of vertices whose exploration has been completed; $U$, the set of vertices currently being explored, kept in a queue (first-in-first-out discipline); $T$, the set of vertices which have yet been processed; $J$, the set of vertices found to fall outside the random set of vertices $V_p$. We initialize $S, J, U = \emptyset$; and $T=V$, and run the algorithm until $U\cup T=\emptyset$.

We say that we are in round $t\in \mathbb{N}$ of the algorithm if exactly $t-1$ vertices have already been queried by the algorithm. In round $t$, we do the following: if $U$ is empty, we consider the first vertex $v$ in $T$ according to the order $\sigma$. If $X_t = 1$, we retain this vertex and move it from $T$ to $U$. If $X_t = 0$, we move it to $J$ and finish the round. Otherwise, if $U\neq \emptyset$, we consider the first vertex $v$ in $U$, and query the first neighbor of $v$ in $T$, according to the order $\sigma$. If $X_t = 1$, we retain this vertex and move it from $T$ to $U$. If $X_t = 0$, we move it to $J$ and finish the round. If there are no neighbors of $v$ in $T$ left to query, we move $v$ from $U$ to $S$ and start the next round, unless $U\cup T=\emptyset$.

When the execution terminates, the set $S$ has the same distribution as $V_p$ and the algorithm has revealed the connected components of $G[V_p]$. Furthermore, at any stage of the algorithm, it has already been revealed that the graph $G$ has no edges between the current set $S$ and the current set $T$, and thus $N_G(S) \subseteq U \cup J$.

\subsection{Auxiliary Lemmas}

We will make use of the following two fairly standard probability bounds. The first is a Chernoff-type tail bound for the binomial distribution (see, for example, Appendix A in \cite{zbMATH06566409}):
\begin{claim} \label{cher}
    Let $X\sim Bin(n,p)$. Then for any $0<t\leq \frac{np}{2}$,
    \begin{align*}
        \mathbb{P}(|X-np|\geq t)\leq 2\exp \left(-\frac{t^2}{3np}\right).
    \end{align*}
\end{claim}

The second is a variant of the well-known Azuma-Hoeffding inequality (see, for example, Chapter 7 in \cite{zbMATH06566409} and \cite[Corollary 6]{zbMATH06788897}):
\begin{claim} \label{hoef}
    Let $m \in N$ and let $p \in [0, 1]$. Let $X = (X_1, X_2, \dots , X_m)$ be a random vector with range $\{0, 1\}^m$ with $X_i$ being independent $Bernoulli(p)$ for every $1\le i\le m$. Let $f:\{0, 1\}^m \xrightarrow{} \mathbb{R}$ be such that there exists $C \in \mathbb{R}$ such that for every $x, x' \in \{0, 1\}^m$ which differ only in one coordinate,
    $$|f(x)-f(x')|\leq C.$$
    Then, for every $t \geq 0$,
    \begin{enumerate}
        \item \begin{align*}
            \mathbb{P}(|f(X)-\mathbb{E}[f(X)]|\geq t)\leq 2\exp \left(-\frac{t^2}{2C^2mp+Ct}\right);
        \end{align*}
        \item \begin{align*}
            \mathbb{P}(|f(X)-\mathbb{E}[f(X)]|\geq t)\leq 2\exp \left(-\frac{t^2}{2C^2m}\right).
        \end{align*}
    \end{enumerate}
\end{claim}

We also state a fairly basic combinatorial identity involving binomial combinatorics. Let $1\le k\le n$ be integers. Then,
\begin{equation}\label{identity}
    \left(\frac{n}{k}\right)^k\le \binom{n}{k}\le \sum_{i=0}^k \binom{n}{i}\le \left(\frac{en}{k}\right)^k.
\end{equation}

\subsection{Some General Lemmas on the Percolated Subgraph}

For any $d$-regular graph, we use the fact, stated and proved in \cite[Claim 2.1]{zbMATH07770097}, that typically connected components in the percolated subgraph admit good vertex expansion:

\begin{lemma}\label{expansion}
    Let $G = (V, E)$ be a $d$-regular graph on $n$ vertices. Let $\epsilon>0$ be a sufficiently small constant, and let $p=\frac{1+\epsilon}{d}$. Then, \textbf{whp}, any connected component $S$ of $G[V_p]$ with $|S|>300\log n$ has $|N_G(S)|\ge \frac{9}{10}d|S|$.
\end{lemma}

The goal of the following lemma and its three corollaries, at a heuristic level, is to \textbf{whp} rule out the existence of connected components in $G[V_p]$ of `intermediate' sizes, where this `intermediate' size depends on `how far' the graph has an almost perfect edge expansion. For these four claims, we consider the following setting: let $\epsilon>0$ be a sufficiently small constant. Let $G$ be a $d$-regular graph on $n$ vertices such that for every $U\subseteq V(G)$ with $|U|\le t$, $e(U,U^c)\ge \left(1-\frac{\epsilon^2}{1000}\right)d|U|$, for some integer $t\coloneqq t(n,d)\le n$. Finally, set $p=\frac{1+\epsilon}{d}$.

\begin{lemma}\label{nsc_general_general}
    Let $A\subseteq V(G)$, and fix $k\in \left[\frac{50}{\epsilon^2}, \frac{t}{4d}\right]$. The probability that
    \begin{align*}
        \left|\bigcup_{u\in A} C_{G[V_p]}(u)\right|=k
    \end{align*}
    is at most $3\epsilon kd\cdot \exp\left(-\frac{\epsilon^2}{25}\cdot k\right)$.
\end{lemma}

\begin{proof}
    We run the BFS algorithm equipped with an order $\sigma$ on $V(G)$, satisfying that the vertices in $A$ are the first vertices in the order.
    After round $r>0$ of the algorithm, we let $V_r=(v_1,\dots,v_r)$ be the vertices that have already been queried by the algorithm, where the order of $V_r$ is according to the query order. Furthermore, let $S_r=V_r\cap V_p$; $S_0=\emptyset$. Observe that by construction, $V_r\subseteq A\cup N_G[S_r]$.

    Suppose that at round $r$ of the algorithm ($1\le r \le n$), we have just finished exploring $\bigcup_{u\in A} C_{G[V_p]}(u)$ and have that $\left|\bigcup_{u\in A} C_{G[V_p]}(u)\right|=k$. In this case, $|S_r|=k$ and all vertices in $N_G[S_r]$ have been queried, thus $|N_G[S_r]|\le r$. Define $I=\left[\frac{k}{\left(1+\frac{3\epsilon}{5}\right)p},\frac{k}{\left(1-\frac{3\epsilon}{5}\right)p}\right]$ and distinguish between three regimes: $1\le r\le \min I$, $r\in I$ and $\max I\le r\le n$. Notice that

    \begin{align}\label{maineq}
    \notag
        \mathbb{P}\left(\left|\bigcup_{u\in A} C_{G[V_p]}(u)\right|=k\right)&\le \mathbb{P}\bigg(\exists_{1\le r\le n}\;\big[|S_r|=k\wedge |N_G[S_r]|\le r\big]\bigg)\\\notag
        &\le \mathbb{P}\left(\exists_{1\le r\le \min I}\;|S_r|=k\right)+\mathbb{P}\left(\exists_{r\in I}\;|N_G[S_r]|\le r\right)+\mathbb{P}\left(\exists_{\max I\le r\le n}\;|S_r|=k\right)\\
        &\le \mathbb{P}\left(|S_{\min I}|\ge k\right)+\sum_{r\in I}\mathbb{P}\left(|N_G[S_r]|\le r\right)+\mathbb{P}\left(|S_{\max I}|\le k\right),
    \end{align}
    where the last inequality holds by the following observation: note that $|S_1|$ is less than $k$. Thus, if $|S_r|=k$ for some $1\le r \le \min I$, then in particular $|S_{\min I}|\ge k$. Hence, the event $\left(\exists_{1\le r\le \min I}\;|S_r|=k\right)$ is contained in the event $\left(|S_{\min I}|\ge k\right)$. Similarly, for $\max I\le r\le n$, the event $\left(\exists_{\max I\le r\le n}\;|S_r|=k\right)$ is contained in the event $\left(|S_{\max I}|\le k\right)$.

    We first bound the probability that $|S_{\min I}|\ge k$ from above. By Claim \ref{cher},
    
    \begin{align}\label{start}
    \notag
    \mathbb{P}\left(|S_{\min I}|\ge k\right)&\le\mathbb{P}\left(Bin\left(\min I,p\right)\ge k\right)\le 2\exp\left(-\frac{\left(k-\min I\cdot p\right)^2}{3\cdot \min I\cdot p}\right)\\
    &=2\exp\left(-\frac{\left(1-\frac{1}{1+\frac{3\epsilon}{5}}\right)^2k^2}{3\cdot \frac{k}{1+\frac{3\epsilon}{5}}}\right)\le 2\exp\left(-\frac{2\epsilon^2}{25}k\right)\le \exp\left(-\frac{\epsilon^2}{25}k\right),
    \end{align}
    where the last inequality holds since $k\ge \frac{50}{\epsilon^2}$.

    We now turn to estimating the probability that $|S_{\max I}|\le k$. By Claim \ref{cher},
    
    \begin{align}\label{end}
    \notag
    \mathbb{P}\left(|S_{\max I}|\le k\right)&\le\mathbb{P}\left(Bin\left(\max I,p\right)\le k\right)\le 2\exp\left(-\frac{\left(\max I\cdot p-k\right)^2}{3\cdot \max I\cdot p}\right)\\
    &=2\exp\left(-\frac{\left(\frac{1}{1-\frac{3\epsilon}{5}}-1\right)^2k^2}{3\cdot \frac{k}{1-\frac{3\epsilon}{5}}}\right)\le 2\exp\left(-\frac{2\epsilon^2}{25}k\right)\le \exp\left(-\frac{\epsilon^2}{25}k\right),
    \end{align}
    where the last inequality holds since $k\ge \frac{50}{\epsilon^2}$. 
    
    Finally, we turn to the case where $r\in I$. In this case we bound from above the probability that the exposure of the components of $A$ is completed in $r$ rounds of the BFS algorithm and that the union of these components consists of $k$ vertices. In particular, it suffices to bound from above the probability that $|N_G[S_r]|\le r$ where $|S_r|=k$. First note that if $|A|\ge 2\cdot \frac{k}{p}$, then by Claim \ref{cher}, with probability at least $1-2\exp\left(-\frac{k}{6}\right)$, we have $|A\cap V_p|>k$, and hence 
    \begin{align*}
        \mathbb{P}\left(\left|\bigcup_{u\in A} C_{G[V_p]}(u)\right|=k\right)\le \mathbb{P}\left(|A\cap V_p|\le k\right)\le 2\exp\left(-\frac{k}{6}\right).
    \end{align*}
    Thus, we can narrow our discussion to sets $A$ of size at most $2\cdot \frac{k}{p}\le 2\left(1+\frac{3\epsilon}{5}\right)r$. Hence,

    \begin{align}\label{neighborhood}
        |A\cup N_G[S_r]|\le |A|+r\le (3+2\epsilon)r.
    \end{align}

    Let $i\le r$, and consider the $i$-th round of the BFS algorithm. We say that a vertex $u\in V(G)$ is \textit{bad} with respect to $S_{i-1}$ if it has at most $\left(1-\frac{\epsilon}{20}\right)d$ neighbors in $V(G)\backslash N_G[S_{i-1}]$, and is \textit{good} otherwise. Once a vertex is labeled as good or bad, the label persists for the duration of the algorithm.
    Assume towards a contradiction that $V_r$ contains more than $\frac{\epsilon}{10} r$ bad vertices. As mentioned earlier, $A\cup N_G[S_r]\supseteq V_r$, thus using \eqref{neighborhood} we have
    \begin{align}\label{upper}
        e(A\cup N_G[S_r])\ge \frac{\epsilon}{10}r\cdot \frac{\epsilon}{20} d\ge \frac{\epsilon^2}{200(3+2\epsilon)}d|A\cup N_G[S_r]|.
    \end{align}
    On the other hand, since $r\in I$, then $|A\cup N_G[S_r]|\le |A|+r\le (3+2\epsilon)r\le 4dk\le t$, and therefore we can use our assumption that for every $U\subseteq V(G)$ with $|U|\le t$, $e(U,U^c)\ge \left(1-\frac{\epsilon^2}{1000}\right)d|U|$ to obtain:
    \begin{align*}
        e(A\cup N_G[S_r],(A\cup N_G[S_r])^c)\ge \left(1-\frac{\epsilon^2}{1000}\right)d|A\cup N_G[S_r]|,
    \end{align*}
    which contradicts \eqref{upper}. Thus, $V_r$ cannot contain more than $\frac{\epsilon}{10} r$ bad vertices. We further show that the number of bad vertices in $S_r$ is typically small; to be precise, we show that it is typically smaller than $\frac{7\epsilon}{10}\cdot rp$. Since we have at least $\left(1-\frac{\epsilon}{10}\right)r$ good vertices in $V_r$, the expected number of good vertices in $S_r$ is at least $\left(1-\frac{\epsilon}{10}\right)rp$, and therefore the number of good vertices in $V_r$ is stochastically dominated by $Bin\left(\left(1-\frac{\epsilon}{10}\right)r,p\right)$. Hence, by Claim \ref{cher}, the probability that we have less than $\left(1-\frac{7\epsilon}{10}\right)rp$ good vertices in $S_r$ is at most:

    \begin{align*}
        \mathbb{P}\left(Bin\left(\left(1-\frac{\epsilon}{10}\right)r,p\right)<\left(1-\frac{7\epsilon}{10}\right)rp\right)&\le 2\exp\left(-\frac{\frac{9\epsilon^2}{25} r^2p^2}{3rp}\right)\\
        &=2\exp\left(-\frac{3\epsilon^2}{25}\cdot rp\right)\\
        &\le \exp\left(-\frac{\epsilon^2}{25}\cdot k\right),
    \end{align*}
    where the last inequality holds since $r\ge\frac{k}{\left(1+\frac{3\epsilon}{5}\right)p}$ and since $k\ge \frac{50}{\epsilon^2}$.
    
    To conclude, with probability at least $1-\exp\left(-\frac{\epsilon^2}{25}\cdot k\right)$, $S_r$ contains at least $\left(1-\frac{7\epsilon}{10}\right)rp$ good vertices. Since every good vertex that falls into $S_r$ contributes at least $\left(1-\frac{\epsilon}{20}\right)d$ new vertices to the closed neighborhood of $S_r$, using $|S_r|\le \left(1+\frac{3\epsilon}{5}\right)rp$, we derive that with probability $1-\exp\left(-\frac{\epsilon^2}{25}\cdot k\right)$,
    \begin{align*}
        |N_G(S_r)|&\ge \left(1-\frac{7\epsilon}{10}\right)rp\cdot \left(1-\frac{\epsilon}{20}\right)d-|S_r|\ge \left(1-\frac{7\epsilon}{10}\right)\left(1-\frac{\epsilon}{20}\right)(1+\epsilon)r-\left(1+\frac{3\epsilon}{5}\right)rp>r,
    \end{align*}
    where the last inequality holds since $d=\omega(1)$ and since $\epsilon>0$ is small enough.
    Hence, 
    \begin{align*}
        \mathbb{P}\left(|N_G[S_r]|\le r\right)\le \exp\left(-\frac{\epsilon^2}{25}\cdot k\right).
    \end{align*}
    Overall,
    \begin{align}\label{middle}
        \sum_{r\in I}\mathbb{P}\left(|N_G[S_r]|\le r\right)\le 2\epsilon kd\cdot \exp\left(-\frac{\epsilon^2}{25}\cdot k\right).
    \end{align}
    
    Putting \eqref{start}, \eqref{end} and \eqref{middle} together, we have, by \eqref{maineq} that 
    \begin{align*}
        \mathbb{P}\left(\left|\bigcup_{u\in A} C_{G[V_p]}(u)\right|=k\right)
        &\le \exp\left(-\frac{\epsilon^2}{25}\cdot k\right) + 2\epsilon kd\cdot \exp\left(-\frac{\epsilon^2}{25}\cdot k\right) + \exp\left(-\frac{\epsilon^2}{25}\cdot k\right)\\
        &\le 3\epsilon kd\cdot \exp\left(-\frac{\epsilon^2}{25}\cdot k\right),
    \end{align*}
    completing the lemma.
\end{proof}

We will utilize Lemma \ref{nsc_general_general} in different ways throughout the proof of Theorems \ref{general} and \ref{n,d,lambda}. The following three corollaries are implementations of Lemma \ref{nsc_general_general}.

First, by taking $|A|=1$, we get the following two statements:

\begin{corollary}\label{nsc_pre}
    \textbf{Whp}, for every vertex $v\in V(G)$, 
        \begin{align*}
            |C_{G[V_p]}(v)|\notin \left[\frac{100}{\epsilon^2}\log n, \frac{t}{4d}\right].
        \end{align*}
\end{corollary} 

\begin{proof}
    Indeed, by the union bound and Lemma \ref{nsc_general_general},
    \begin{align*}
        \mathbb{P}\left(\exists_{v\in V(G)}\; \left|C_{G[V_p]}(v)\right|\in \left[\frac{100}{\epsilon^2}\log n, \frac{t}{4d}\right]\right)&\le n\cdot \sum_{k=\frac{100}{\epsilon^2}\log n}^{\frac{t}{4d}} 3\epsilon kd\cdot \exp\left(-\frac{\epsilon^2}{25}\cdot k\right)\\
        &\le n\cdot 3\epsilon \cdot \frac{t}{4d}\cdot d\cdot \frac{1}{n^4}=o(1),
    \end{align*}
    which establishes the corollary (the last equality holds since $t\le n$). 
\end{proof}

\begin{corollary}\label{nsc1_pre}
    Fix $v\in V(G)$. Then, \textbf{whp}, 
        \begin{align*}
            |C_{G[V_p]}(v)|\notin \left[\sqrt{d}, \frac{100}{\epsilon^2}\log n\right].
        \end{align*}
\end{corollary}

\begin{proof}
By Lemma \ref{nsc_general_general},
    \begin{align*}
        \mathbb{P}&\left( |C_{G[V_p]}(v)|\in \left[\sqrt{d},\frac{100}{\epsilon^2}\log n\right]\right)\le \sum_{k=\sqrt{d}}^{\frac{100}{\epsilon^2}\log n} 3\epsilon kd\cdot \exp\left(-\frac{\epsilon^2}{25}\cdot k\right)\\
        &=3\epsilon d \sum_{k=\sqrt{d}}^{\frac{100}{\epsilon^2}\log n}\exp\left(\log k-\frac{\epsilon^2}{25}\cdot k\right)\le 3\epsilon d \sum_{k=\sqrt{d}}^{\infty}\exp\left(\left(\epsilon^3-\frac{\epsilon^2}{25}\right)\cdot k\right)\\
        &\le 3\epsilon d\cdot \frac{\exp\left(\epsilon^3-\frac{\epsilon^2}{25}\right)^{\sqrt{d}}}{1-\exp\left(\epsilon^3-\frac{\epsilon^2}{25}\right)}=\frac{3\epsilon}{1-\exp\left(\epsilon^3-\frac{\epsilon^2}{25}\right)}\cdot d\exp\left(\left(\epsilon^3-\frac{\epsilon^2}{25}\right)\sqrt{d}\right)=o(1),
    \end{align*}
    where the second inequality and the last equality hold as $d=\omega(1)$.
\end{proof}

We also utilize Lemma \ref{nsc_general_general} for a specific type of sets -- balls around vertices of $G$.

\begin{corollary}\label{nsc_ball_pre}
    Let $r\ge 0$ be an integer. Then, \textbf{whp}, for every vertex $v\in V(G)$,
        \begin{align*}
            \left|\bigcup_{u\in B(v,r)} C_{G[V_p]}(u)\right|\notin \left[\frac{100}{\epsilon^2}\log n, \frac{t}{4d}\right].
        \end{align*}
\end{corollary}

\begin{proof}
    Indeed, by the union bound and Lemma \ref{nsc_general_general},
    \begin{align*}
        \mathbb{P}\left(\exists_{v\in V(G)}\; \left|\bigcup_{u\in B(v,r)} C_{G[V_p]}(u)\right|\in \left[\frac{100}{\epsilon^2}\log n, \frac{t}{4d}\right]\right)&\le n\cdot \sum_{k=\frac{100}{\epsilon^2}\log n}^{\frac{t}{4d}} 3\epsilon kd\cdot \exp\left(-\frac{\epsilon^2}{25}\cdot k\right)\\
        &\le n\cdot 3\epsilon \cdot \frac{t}{4d}\cdot d\cdot \frac{1}{n^4}=o(1),
    \end{align*}
    which establishes the corollary (the last equality holds as $t\le n$). 
\end{proof}

\section{Proof of Theorem~\ref{general}} \label{Pgeneral}

Throughout the rest of this section, we assume that $c_1,c_2>0$ are some constants. We will prove Theorem~\ref{general} for $\alpha=1$ in order to present a cleaner argument. The modifications required to make it work for other $0<\alpha<1$ are presented in Section \ref{modi}. 

Let $\epsilon>0$ be a sufficiently small constant. We set $p=\frac{1+\epsilon}{d}$, and the value of $C>0$ and $c_3\coloneqq c_3(c_2,\epsilon)$ will be chosen sufficiently large. We let $G$ be a $d$-regular graph on $n$ vertices satisfying \ref{prop0A}, \ref{prop1A} and \ref{prop2A}. As mentioned after the statement of Theorem~\ref{general}, for assumption \ref{prop2A} to be meaningful, we assume that $d<n^{\gamma}$, for some $\gamma>0$.

Let us first briefly discuss the proof’s strategy. The key ideas of the proof are similar to those of \cite{zbMATH08098276,diskin2025componentslargesmalli}. We utilize a double-exposure/sprinkling argument \'{a} la Ajtai-Koml\'{o}s-Szemer\'{e}di \cite{zbMATH03769676}. Let $s\coloneqq s(\epsilon)>0$ be a small enough constant. Define $p_2=\frac{s}{d}$ and let $p_1$ be such that $(1-p_1)(1-p_2)=1-p$, noting that $p_1\geq \frac{1+\epsilon-s}{d}$. Notice that $G[V_p]$ has the same distribution as $G[V_{p_1}\cup V_{p_2}]$. We abbreviate $G_1:=G[V_{p_1}]$ and $G_2:=G[V_{p_1}\cup V_{p_2}]$. 

First, we use Corollary \ref{nsc_pre} to show that typically there are no components in $G_1$ (nor in $G_2$) whose size is between $\frac{100}{\epsilon^2}\log n$ and $d^5\log^C n$.
Then, we show in Lemma \ref{one_comp_2} that typically all components of size $\Omega(\log n)$ in $G_1$ merge after sprinkling with probability $p_2$.
We also show that when sprinkling with probability $p_2$, we typically do not ‘accidentally’ merge small components (of size $O(\log n)$) in $G_1$ into a largish component (of size $\Omega(\log n)$) in $G_2$ (Lemma \ref{big_intersect_W}). Finally, we show in Lemma \ref{large_comp_2} that the giant component in $G_2$ is of the `right' order (as stated in Theorem~\ref{general}).

\subsection{Components are Big or Small}

Our first goal is to apply the corollaries of Lemma \ref{nsc_general_general} on our graph $G$, showing that, \textbf{whp}, $G[V_p]$ contains no connected components of `intermediate' sizes. This is indeed the case, since our graph $G$ satisfies \ref{prop2A}. 
Let us state the claims here as they apply to our graph $G$. We formulate these arguments for $G[V_p]$, noting that it applies to $G[V_{p_1}]$ as well due to the fact that $dp_1>1$, which is still supercritical.

\begin{lemma}\label{nsc_general}
    Let $A\subseteq V(G)$, and fix $k\in \left[\frac{50}{\epsilon^2}, d^5\log^C n\right]$. The probability that
    \begin{align*}
        \left|\bigcup_{u\in A} C_{G[V_p]}(u)\right|=k
    \end{align*}
    is at most $3\epsilon kd\cdot \exp\left(-\frac{\epsilon^2}{25}\cdot k\right)$.
\end{lemma}

\begin{corollary}\label{nsc}
    \textbf{Whp}, for every vertex $v\in V(G)$, 
        \begin{align*}
            |C_{G[V_p]}(v)|\notin \left[\frac{100}{\epsilon^2}\log n, d^5\log^C n\right].
        \end{align*}
\end{corollary} 

\begin{corollary}\label{nsc1}
    Fix $v\in V(G)$. Then, \textbf{whp}, 
        \begin{align*}
            |C_{G[V_p]}(v)|\notin \left[\sqrt{d}, \frac{100}{\epsilon^2}\log n\right].
        \end{align*}
\end{corollary}

\begin{corollary}\label{nsc_ball}
    Let $r\ge 0$ be an integer. Then, \textbf{whp}, for every vertex $v\in V(G)$,
        \begin{align*}
            \left|\bigcup_{u\in B(v,t)} C_{G[V_p]}(u)\right|\notin \left[\frac{100}{\epsilon^2}\log n, d^5\log^C n\right].
        \end{align*}
\end{corollary}

\subsection{Big components merge after sprinkling}\label{need p2}

In this subsection, we show that typically all components of size $\Omega(d^5\log^C n)$ in $G_1$ merge after sprinkling with probability $p_2$. This, together with Corollary \ref{nsc}, will imply that all components of size $\Omega(\log n)$ actually merge. To that end, we define the set of vertices in ‘large’ components in $G_1$ and in $G_2$:
\begin{align*}
    W_1=\bigg\{v\in V(G)\;\bigg|\; |C_{G_1}(v)|\geq \frac{100}{\epsilon^2}\log n\bigg\},
\end{align*}
and

\begin{align*}
    W_2=\bigg\{v\in V(G)\;\bigg|\; |C_{G_2}(v)|\geq \frac{100}{\epsilon^2}\log n\bigg\}.
\end{align*}
We note that $W_1\subseteq W_2$.

The following two lemmas show that, typically, every vertex $v\in V(G)$ is close to a `large' component from $W_1$.

\begin{lemma} \label{dist_2}
    For every $v\in V(G)$, $|B_G(v,2\log_d\log n+1)|\geq c_2c_3d\log n$.
\end{lemma}

\begin{proof}
    Let $v\in V(G)$ and $r\in \mathbb{N}$. Denote $U=B_G(v,r)$. Note that $U\cup N_G(U)=B_G(v,r+1)$. 
    By \ref{prop1A}, $$|B_G(v,r+1)|\geq \min \Bigg\{c_3\log n, c_2d|B_G(v,r)|\Bigg\}.$$
    Therefore, for $r_0=\log_{c_2d}(c_3\log n)\leq 2\log_d\log n$, we have $|B_G(v,r_0)|\geq c_3\log n$, by the definition of $r_0$. Furthermore, since we can always find a subset of size $c_3\log n$ in $B_G(v,r_0)$, which contains $v$, then again by \ref{prop1A}, $|B_G(v,r_0+1)|\ge c_2c_3d\log n$.
\end{proof}

\begin{lemma} \label{W_2}
    \textbf{Whp}, for every $v\in V(G)$, $dist_G(v,W_1)\leq 2\log_d\log n+1$.
\end{lemma}
\begin{proof}
    By Lemma \ref{dist_2}, \textbf{whp}, for every $v\in V(G)$, $|B(v,2\log_d\log n+1)|\geq c_2c_3d\log n$. Further observe that on the other hand,
    \begin{align*}
        |B(v,2\log_d\log n+1)|\le \sum_{i=0}^{2\log_d\log n+1} d^i=\frac{d^{2\log_d\log n+2}-1}{d-1}\le 2d\log^2 n.
    \end{align*}
    
    We now show that \textbf{whp}, for every $v\in V(G)$, $B(v,2\log_d\log n+1)\cap W_1\neq \emptyset$. Indeed, by Corollary \ref{nsc_ball}, with $r=2\log_d\log n+1$, \textbf{whp},
    \begin{align*}
        \left|\bigcup_{u\in B(v,2\log_d\log n+1)}C_{G_1}(u)\right|\notin \left[\frac{100}{\epsilon^2}\cdot\log n,d^5\log^C n\right].
    \end{align*}
    Furthermore, since $|B(v,2\log_d\log n+1)|\geq c_2c_3d\log n$, and since $c_3=c_3(c_2,\epsilon)$ is large enough, then by Claim \ref{cher}, \textbf{whp}, the number of vertices which remains from $B(v,2\log_d\log n+1)$ after site percolation is at least $c_2c_3\log n/2\ge \frac{100}{\epsilon^2}\cdot\log n$. Hence, \textbf{whp}, 
    \begin{align*}
        \left|\bigcup_{u\in B(v,2\log_d\log n+1)}C_{G_1}(u)\right|\ge d^5\log^C n.
    \end{align*}
    Then, since $|B(v,4\log_d\log n+1)|\le 2d\log^2 n$, by the pigeonhole principle, there exists $u\in B(v,2\log_d\log n+1)$, such that $\left|C_{G_1}(u)\right|\ge \frac{d^5\log^C n}{2d\log^2 n}\ge \frac{100}{\epsilon^2}\log n$. As such, \textbf{whp} there is a component of size at least $\frac{100}{\epsilon^2}\log n$ in $G_1$ at distance at most $2\log_d \log n+1$ from every vertex $v\in V(G)$. By the definition of $W_1$, we establish the lemma.
\end{proof}

Finally, since each vertex $v\in V(G)$ is at distance at most $2\log_d\log n+1$ from $W_1$, we are now ready to argue that after sprinkling with probability $p_2$, \textbf{whp} all components of $W_1$ merge.

\begin{lemma} \label{one_comp_2}
    \textbf{Whp} there is a component $K$ in $G_2$ such that $W_1\subseteq V(K)$.
\end{lemma}

\begin{proof}
    Applying Corollary \ref{nsc} yields that \textbf{whp}, all components in $G_1[W_1]$ are of size at least $d^5\log^C n$. Further, by Lemma \ref{W_2}, \textbf{whp} every $v\in V(G)$ is at distance at most $2\log_d\log n+1$ from some $w\in W$. We continue assuming these properties hold deterministically.
    It suffices to show that \textbf{whp} for every partition of $W_1$ into two $G_1$-component-respecting parts $A$ and $B$, with $a = |A| \leq |B|$, there exists a path in $G_2$ between $A$ and $B$. First, we bound from below the number of vertex-disjoint paths between $A$ and $B$. To that end, let $A'$ be $A$ together with all the vertices in $V (G)\backslash B$ that are at distance at most $2\log_d\log n+1$ from some vertex in $A$. Similarly, let $B'$ be $B$ together with all the vertices in $V(G)\backslash A'$ that are at distance at most $2\log_d\log n+1$ from some vertex in $B$. Note that $V(G)= A' \sqcup B'$, and thus by \ref{prop0A}, $e_G(A',B')=e_G(A',V(G)\backslash A')\geq c_1\cdot \min\{|A'|,|V(G)\backslash A'|\}\ge c_1|A|$. In this way we built at least $c_1a$ paths between $A$ and $B$, passing through $A'$ and $B'$ and using the edges in $e_G(A',B')$, of length at most $2(2\log_d\log n+1)+1=4\log_d\log n+3$. As $G$ is $d$-regular, we can find at least $\frac{c_1a}{d^{4 \log_d \log n+3}}=\frac{c_1a}{d^3\log^4 n}$ vertex-disjoint (disjoint outside of $A$ and $B$) paths (in $G$) of length at most $4 \log_d \log n+3$ between $A$ and $B$. Notice that each path contains $4\log_d\log n+3+1=4\log_d\log n+4$ vertices, and that the ends of each path are in $W_1$, i.e., two of these vertices are already in $V(G_1)$. Hence, the probability that such path falls into $G_2$ is at least $p_2^{4\log_d\log n+2}$.
    
    We now bound from above the number of component-respecting partitions of $W_1$ into $W_1= A \sqcup B$ with $|A| = a$. Since every component of $W_1$ is of size at least $d^5\log^C n$ and since $W_1= A \sqcup B$ is a component-respecting partition, given that $|A| = a$ there are at most 
    \begin{align*}
        \sum_{i=1}^{a/(d^5\log^C n)}\binom{n/(d^5\log^C n)}{i}\le\left(\frac{en}{a}\right)^\frac{a}{d^5\log^C n}\leq n^\frac{a}{d^5\log^C n}
    \end{align*}
    ways to choose $A$ (and hence the partition). Thus, by the union bound, the probability that there exists such a partition $(A,B)$ without a path in $G_2$ between $A$ and $B$ is at most
    \begin{align*}
        \sum_{a=d^5\log^C n}^{n/2}& n^{\frac{a}{d^5\log^C n}}\left(1-p_2^{4\log_d\log n+2}\right)^{\frac{c_1a}{d^3\log^4 n}}\le \sum_{a=d^5\log^C n}^{n/2} n^{\frac{a}{d^5\log^C n}} \exp \left(-\frac{c_1a}{d^3\log^4 n}p_2^{4\log_d\log n+2}\right)\\
        &\le \sum_{a=d^5\log^C n}^{n/2} \exp \left[a\left(\frac{\log n}{d^5\log^C n}-\frac{c_1}{d^5\log^8 n}s^{4\log_d\log n+2}\right)\right].
    \end{align*}
    Notice that since $d=\omega(1)$, we have $s^{4\log_d\log n}=(\log n)^{4\frac{\log s}{\log d}}\ge \frac{1}{\log n}$, and hence:
    \begin{align*}
        \sum_{a=d^5\log^C n}^{n/2}& \exp \left[a\left(\frac{\log n}{d^5\log^C n}-\frac{c_1}{d^5\log^8 n}s^{4\log_d\log n+2}\right)\right]\\
        &\le n\cdot \exp \left[d^5\log^C n\left(\frac{\log n}{d^5\log^C n}-\frac{c_1s^2}{d^5\log^9 n}\right)\right]\\
        &\le n\cdot \exp \left(\log n-c_1s^2\cdot \log^{C-9} n\right)=o(1),
    \end{align*}
    where the last equality follows from an appropriate choice of $C$. It should be noted that the power $6$ (of $d$) in \ref{prop1A} plays a role in this calculation.
\end{proof}

\subsection{Sprinkling does not create new large components}

By Corollary \ref{nsc}, \textbf{whp} there are no components in $G_1$, nor in $G_2$, of size is between $\frac{100}{\epsilon^2}\log n$ and $d^5\log^C n$. Further, using Corollary \ref{nsc} and Lemma \ref{one_comp_2} we can derive that \textbf{whp} all components in $G_1$ of size was at least $\frac{100}{\epsilon^2}\log n$ merge into a unique component in $G_2$. Note, however, that we still need to rule out the existence of components outside of $W_1$ of size is at least $d^5\log^C n$ in $G_2$.

\begin{lemma} \label{big_intersect_W}
    \textbf{Whp}, there is no component in $G_2$ of size at least $d^5\log^C n$, which does not intersect $W_1$.
\end{lemma}

\begin{proof}
    For $v\in V(G)$, let $A_v$ be the event that $C_{G_2}(v)\cap W_1=\emptyset$ and $|C_{G_2}(v)|\geq \log^2 n$. By Corollary \ref{nsc}, we have that the probability of an event violating the statement of the Lemma is at most $\mathbb{P}\left[\cup_{v\in V(G)} A_v\right]+o(1)$.
    
    Let $(X_i)_{i=1}^{|V|}$ be a sequence of i.i.d. $Bernoulli(p_1)$ random variables, and let $(Y_i)_{i=1}^{|V|}$ be a sequence of i.i.d. $Bernoulli(p_2)$ random variables. We will run a BFS-type algorithm, very similar to the one in Section \ref{pre}, however, we will utilize both sequences of random variables.
    
    Let $\sigma$ be an arbitrary ordering of $V(G)$. Now, as long as $|S\cup U|\leq \log^2 n$, when we query the $t$-th vertex, we move it into $U$ if $X_t = 1$ or $Y_t = 1$, and move it into $J$ only if $X_t = Y_t = 0$ ($S,U,J$ as in the BFS algorithm described in Section \ref{pre}). Once $|S \cup U| = \log^2 n$, and at any subsequent round of the algorithm, when we query the $t$-th vertex, we move it into $U$ if $X_t = 1$, and move it into $J$ if $X_t = 0$. Note that in this manner, up until $|S \cup U| = \log^2 n$, the exploration of the component of $v$ is in $G_2$, and afterwards we continue the exploration in $G_1$. Also, each query gets a positive answer independently and with probability at least $p_1$.

    Observe that if $A_v$ occurs for some $v\in V(G)$, then we must have reached a moment where $|S \cup U| = \log^2 n$. Suppose that $C_{G_2}(v) \cap W_1 = \emptyset$. Then, the above process must have ended before $|S \cup U| = \log^4 n$. Indeed, otherwise, there are vertices in $C_{G_2}(v)$ belonging to components of size at least $\log^4 n/ \log^2 n = \log^2 n$ in $G_1$, and thus (by definition) intersecting with $W_1$. Hence, the probability that $A_v$ occurs is at most the probability that the above process stopped (and as a result, $U$ empties) at some round $t$ of the algorithm where $\log^2 n \leq |S| \leq \log^4 n$. On the other hand, notice that the above process is dominated by a BFS process with probability $p$ and dominates a BFS process with probability $p_1$, hence using Lemma \ref{nsc_general} in each case implies that the probability that $\log^2 n \leq |S| \leq \log^4 n$ is at most $o\left(\frac{1}{n}\right)$. Thus, by the union bound over all $v\in V(G)$, $\mathbb{P}\left[\cup_{v\in V(G)} A_v\right]=o(1)$, completing the  proof.
\end{proof}

For the last step of the proof of Theorem~\ref{general}, we show that \textbf{whp}, $|W_2|=(1+o(1))x\cdot n$.

\begin{lemma} \label{large_comp_2}
    \textbf{Whp}, $|W_2|=(1+o(1))x\cdot n$.
\end{lemma}

\begin{proof}
    Let us first show that $\mathbb{E}[W_2] = (1 + o(1))x\cdot n$. To that end, fix $v \in V$ and let us estimate $\mathbb{P}\left(|C_{G[V_p]}(v)|\ge \frac{100}{\epsilon^2}\log n\right)$. Run the BFS algorithm rooted at $v$. Since $G$ is $d$-regular, this BFS exploration is stochastically dominated by the vertex version of Galton-Watson tree with offspring distribution $Bin(d,p)$, as expressed in Section \ref{intro}. Since $dp = 1+\epsilon$, standard results (see, for example, \cite{zbMATH06976511}, Theorem 4.3.12) imply that $\mathbb{P}\left(|C_{G[V_p]}(v)|\ge \frac{100}{\epsilon^2}\log n\right)\le (1+o(1))x$.

    On the other hand, consider the BFS exploration with the following alteration --- we terminate the process either once $v$ moved to $S$ (i.e, the process on $C_{G[V_p]}(v)$ has ended), or once we have discovered $\sqrt{d}$ vertices. Then, during the exploration process every vertex in the queue has at least $d-\sqrt{d}$ neighbors in $G$, and thus this BFS exploration stochastically dominates a Galton-Watson tree with offspring distribution $Bin\left(d-\sqrt{d},p\right)$. Since $\left(d -\sqrt{d}\right)p = \lambda - o(1)>1$, we have by standard results that $\mathbb{P}\left(|C_{G[V_p]}(v)|\ge \sqrt{d}\right)\ge (1+o(1))x$, and thus by Corollary \ref{nsc1}, $\mathbb{P}\left(|C_{G[V_p]}(v)|\ge \frac{100}{\epsilon^2}\log n\right)\ge (1+o(1))x$. Overall, $\mathbb{E}[W_2] = (1+o(1))x\cdot n$.
 
    To show that $|W_2|$ is concentrated around its mean, consider the standard vertex-exposure martingale. Each vertex can change the value of $|W_2|$ by at most $\frac{100}{\epsilon^2}d\log n$. Thus, by the first part of Claim \ref{hoef},
    $$\mathbb{P}\left(\big||W_2|-\mathbb{E}[|W_2|]\big|\geq n^{3/4}\right)\leq 2\exp \left(-\Theta\left(\frac{n^{3/2}}{np\cdot d^2\log^2 n}\right)\right)=o(1),$$
    where the last equation holds since $d<n^{\gamma}$ for an appropriate $\gamma>0$. Therefore, \textbf{whp} $|W_2| = (1 + o(1))x\cdot n$, as required.
\end{proof}

We are now ready to prove Theorem~\ref{general}.

\begin{proof}[Proof of Theorem~\ref{general}]\label{finalproofgeneral}
    By Lemma \ref{one_comp_2}, \textbf{whp} there is a unique component $L_1$ in $G_2$, such that $W_1\subseteq V(L_1)$. By Corollary \ref{nsc}, \textbf{whp} any component in $G_2$ besides $L_1$ is either of size at most $\frac{100}{\epsilon^2}\log n$, or of size at least $d^5\log^C n$. By Lemma \ref{big_intersect_W}, \textbf{whp} any component in $G_2$ whose size is at least $d^5\log^C n$ intersects with $W_1$, and is thus part of $L_1$. Thus, \textbf{whp} all components of $G_2$ besides $L_1$ are of size at most $\frac{100}{\epsilon^2}\log n$.
    
    Let us now show that \textbf{whp} $|V(L_1)| = (1 + o(1))x\cdot n$, where $x = x(p)$ is defined according to \eqref{eq}. By Lemma \ref{large_comp_2}, \textbf{whp} there are $(1 + o(1))x\cdot n$ vertices in $G[V_p]$ in components whose size is at least $\frac{100}{\epsilon^2}\log n$. Since \textbf{whp} there is only one component $L_1$, whose size in $G[V_p]$ is at least $\frac{100}{\epsilon^2}\log n$, we
    conclude that \textbf{whp} $|V(L_1)| = (1 + o(1))x\cdot n$.
\end{proof}

\subsection{Modification of the proof of Theorem~\ref{general} for $0<\alpha<1$}\label{modi}

The proof of Theorem~\ref{general} for $0<\alpha<1$ is essentially the same as the proof presented in this section. The distinction between the proofs appears mainly in Lemma \ref{dist_2}.

Since the vertex expansion in \ref{prop1A} is slightly weaker than in the case where $\alpha=1$, the radius in Lemma \ref{dist_2} must be increased. This larger radius is required to ensure that the ball around each vertex still contains $\Omega\left(d\log n\right)$ vertices. Specifically, the required adjustments lead to the following upper bound on the radius:

\begin{align*}
    \frac{1}{\alpha}+\frac{2}{\alpha}\log_d\log n.
\end{align*}

One can then observe that Lemma \ref{one_comp_2} follows verbatim with the new radius, under the assumption that the edge expansion in \ref{prop2A} holds for sets of size up to $d^{\frac{4}{\alpha}+2}\log^C n$.
The rest of the proof stays the same, under this modification.

\section{Proof of Theorem~\ref{n,d,lambda}} \label{Pn,d,lambda}

This section addresses $(n,d,\lambda)$-graphs. First, we state the well known ``expander mixing lemma" which is a classical result by Alon and Chung \cite{zbMATH04073036} regarding the edge distribution of $(n, d, \lambda)$-graphs:

\begin{claim} \label{original_eml}
    Let $G = (V, E)$ be an $(n, d, \lambda)$-graph. Then for any disjoint pair of subsets $B,C\subseteq V(G)$,
    \begin{align*}
        \left|e_G(B,C)-\frac{d}{n}|B||C|\right|\le \lambda\sqrt{|B||C|}.
    \end{align*}
\end{claim} 

Let us recall the setting of Theorem~\ref{n,d,lambda}: let $\epsilon>0$ be a sufficiently small constant and let $\delta\coloneqq \delta(\epsilon)>0$ be chosen sufficiently small. We set $p=\frac{1+\epsilon}{d}$. Let $G$ be an $(n,d,\lambda)$-graph such that $d=\omega_n(1)$, $d=o(n)$, and $\frac{\lambda}{d}<\delta$. 

Let us emphasize the similarities and differences between the assumptions in Theorem~\ref{general} and Theorem~\ref{n,d,lambda}. Note that due to Claim \ref{original_eml}, our graph satisfies properties \ref{prop0A} and a version of \ref{prop2A} --- for every $U\subseteq V(G)$ with $|U|\le 4\epsilon^3n$, we have $e_G(U,U^c)\ge \left(1-\frac{\epsilon^2}{1000}\right)d|U|$. Therefore, the steps in the proof of Theorem~\ref{general} which do not use \ref{prop1A} are still valid in our case. In fact, Section \ref{need p2} is the only place where we use \ref{prop1A} in the proof of Theorem~\ref{general}. We show that in our case, Lemma \ref{expansion}, together with Claim \ref{original_eml}, are sufficient to establish an analogous version of Section \ref{need p2}.

Recall our notation from the proof of Theorem~\ref{general}: for a small constant $s\coloneqq s(\epsilon)>0$, we set $p_2=\frac{s}{d}$ and let $p_1$ be such that $(1-p_1)(1-p_2)=1-p$, noting that $p_1\geq \frac{1+\epsilon-s}{d}$. Notice that $G[V_p]$ has the same distribution as $G[V_{p_1}\cup V_{p_2}]$. We abbreviate $G_1:=G[V_{p_1}]$ and $G_2:=G[V_{p_1}\cup V_{p_2}]$.
We define the set of vertices in ‘large’ components in $G_1$:
\begin{align*}
    W_1=\bigg\{v\in V(G)\;\bigg|\; |C_{G_1}(v)|\geq \frac{100}{\epsilon^2}\log \left(\frac{n}{d}\right)\bigg\}.
\end{align*}

\begin{lemma} \label{one_comp_2_pseudo}
    \textbf{Whp} there is a component $K$ in $G_2$ such that $W_1\subseteq V(K)$.
\end{lemma}

\begin{proof}
    First, since $|V_{p_1}|\sim Bin(n,p_1)$, then using Claim \ref{cher}, we have that \textbf{whp}, $|V_{p_1}|\le 2\cdot \frac{n}{d}$. Now, applying Corollary \ref{nsc_pre} with $t=4\epsilon^3n$ (which is valid by Claim \ref{original_eml}, as mentioned earlier in this section) yields that \textbf{whp}, $G_1$ contains no components of size between $\frac{100}{\epsilon^2}\log n$ and $\epsilon^3\cdot \frac{n}{d}$. By adapting and refining the proofs of Lemma \ref{nsc_general_general} and Corollary \ref{nsc_pre}, we can push this range forward to establish that, in fact, \textbf{whp} there are no components in $G_1$ of size between $\frac{100}{\epsilon^2}\log \left(\frac{n}{d}\right)$ and $\epsilon^3\cdot \frac{n}{d}$ (indeed, note that the number of vertices in $G[V_{p_1}]$ is well very concentrated around $\Theta\left(\frac{n}{d}\right)$) and hence all components in $G_1[W_1]$ are \textbf{whp} of size at least $\epsilon^3\cdot \frac{n}{d}$. Since \textbf{whp} $|V_{p_1}|\le 2\cdot \frac{n}{d}$, then \textbf{whp} there are at most $\frac{2}{\epsilon^3}$ such components. We aim to show that \textbf{whp} for every pair $A,B$ of components in $G_1[W_1]$, there exists a path in $G_2$ between $A$ and $B$. Since \textbf{whp}  there are at most $\frac{2}{\epsilon^3}$ such components, it suffices to prove the argument for a specific pair, that is, we show that given a pair $A,B$ of components in $G_1[W_1]$, then, \textbf{whp}, there exists a path in $G_2$ connecting $A$ and $B$.

    Let $A,B$ be two components in $G_1[W_1]$. Since $|A|,|B|\ge \epsilon^3\cdot \frac{n}{d}$, then by Lemma \ref{expansion}, $|N_G(A)|,|N_G(B)|\ge \frac{9\epsilon^3}{10}n$. Notice that we can assume that $|N_G(A)\cap N_G(B)|< \frac{\epsilon^3}{5} n$, as otherwise \textbf{whp} $|N_G(A)\cap N_G(B)\cap V_{p_2}|\ge 1$ which leads to a path of length $2$ between $A$ and $B$ in $G_2$. We define $A'=N_G(A)$, and $B'=N_G(B)\backslash A'$. Since $|N_G(A)|,|N_G(B)|\ge \frac{9\epsilon^3}{10}n$, and $|N_G(A)\cap N_G(B)|< \frac{\epsilon^3}{5} n$, we have $|A'|,|B'|\ge \frac{\epsilon^3}{2}n$. Now, by the expander mixing lemma (Claim \ref{original_eml}):
    
    \begin{align*}
        e_G(A',B')\ge \frac{d}{n}|A'||B'|-\lambda\sqrt{|A'||B'|}\ge \frac{\epsilon^6}{4}dn-\lambda n\ge \frac{\epsilon^6}{5}dn,
    \end{align*}
    where we use here the fact that $\frac{\lambda}{d}<\delta$ and the fact that $\delta$ is sufficiently small in terms of $\epsilon$. We aim to find, \textbf{whp}, an edge between $A'$ and $B'$ with both endpoints falling into $G_2$, after sprinkling with $p_2$, thereby connecting $A$ and $B$ in $G_2$.

    We expose the vertices of $A'\cup B'$ with probability $p_2$ in two steps. First, let $R=A'\cap V_{p_2}$, and let $B^*\subseteq B'$ be a subset of $B'$ such that $v\in B^*$ if and only if $N_G(v)\cap R\ne \emptyset$. That is, $B^*$ contains the vertices of $B'$ that have at least one neighbor in $R$. Let $v\in B'$. Denote by $d_{A'}(v)\coloneqq |N_G(v)\cap A'|$ the number of neighbors of $v$ that lie in $A'$. Observe that
    \begin{align*}
        \mathbb{P}(v\in B^*)=1-(1-p_2)^{d_{A'}(v)}\ge \frac{d_{A'}(v)\cdot p_2}{2},
    \end{align*}
    where the last inequality holds since $d_{A'}(v)\cdot p_2\le s$ and since $s>0$ is a small constant.
    Hence, 
    \begin{align*}
        \mathbb{E}[|B^*|]\ge \sum_{v\in B'}\frac{d_{A'}(v)\cdot p_2}{2}=\frac{p_2}{2}\cdot \sum_{v\in B'}d_{A'}(v)=\frac{p_2}{2}\cdot e_G(A',B')\ge \frac{\epsilon^6 s}{10}\cdot n.
    \end{align*}
    Now, notice that changing whether $v\in A'$ goes to $R$ influence at most $d_{A'}(v)\le d$ vertices, hence applying Claim \ref{hoef} yields that
    \begin{align*}
        \mathbb{P}\left(|B^*|\le\frac{\mathbb{E}[|B^*|]}{2}\right)&\le 2\exp\left(-\frac{\mathbb{E}[|B^*|]^2}{8d^2|A'|p_2+2d\cdot \mathbb{E}[|B^*|]}\right)\le 2\exp\left(-\frac{\frac{\epsilon^{12}s^2}{100}n^2}{8sdn+2dn}\right)\\
        &=2\exp\left(\frac{-\epsilon^{12}s^2n}{100(8s+2)d}\right)=o(1),
    \end{align*}
    where the last equality holds since $d=o(n)$. Therefore, \textbf{whp}, $|B^*|\ge \frac{\epsilon^6 s}{20}\cdot n$.

    We now expose the vertices of $B'$ with probability $p_2$. By Claim \ref{cher}, we have that \textbf{whp}, $|B^*\cap V_{p_2}|\ge \frac{\frac{\epsilon^6 s}{20}\cdot n\cdot p_2}{2}=\frac{\epsilon^6 s^2}{40}\cdot \frac{n}{d}$. In particular, \textbf{whp} there exists $v\in B'\cap V_{p_2}$ that has a neighbor in $A'\cap V_{p_2}$, which establishes the lemma.
\end{proof}

\section{Proof of Theorem~\ref{constant}}\label{Pconstant}

In this section we give a proof outline of Theorem~\ref{constant}, since the proof of Theorem~\ref{constant} has many similarities to that of Theorem~\ref{general}. The reason for the similarity is that properties \ref{prop0C} and \ref{prop2C} are quite similar to properties \ref{prop0A}-\ref{prop2A}. First, let us restate the setting: let $n$ be a sufficiently large integer, and let $3\le d=O(1)$ be an integer. Fix $1<\alpha<d-1$ and $b>0$ as constants. Set $p=\frac{\alpha}{d-1}$. $\delta:=\delta(\alpha)>0$ and $C:=C(\alpha)>0$ are chosen sufficiently small and large, respectively. Let $G$ be a $d$-regular graph on $n$ vertices, satisfying properties \ref{prop0C}, \ref{prop2C} with $C'\coloneqq C+1$. Similar to the proof of Theorem~\ref{general}, we set $p_2=\frac{s}{d-1}$ for a sufficiently small $s\coloneqq s(\alpha,C)>0$, and let $p_1$ be such that $(1-p_1)(1-p_2)=1-p$, noting that $p_1\geq \frac{\alpha-s}{d-1}$. Notice that $G[V_p]$ has the same distribution as $G[V_{p_1}\cup V_{p_2}]$. We abbreviate $G_1:=G[V_{p_1}]$ and $G_2:=G[V_{p_1}\cup V_{p_2}]$.

Let us show how to translate property \ref{prop2C} to vertex- and edge-expansion properties similar in spirit to properties \ref{prop1A} and \ref{prop2A}. For the latter, it is obvious that property \ref{prop2C} yields that for every $U\subseteq V(G)$ with $|U|\leq \log^{C'} n$, we have $e(U,U^c)\geq (d-2-2\delta)|U|$. For the vertex expansion, see the following lemma:

\begin{lemma} \label{basic2}
    For every $U\subseteq V(G)$ with $|U|\leq \log^C n$, $|N_G(U)|\ge \left(\frac{d}{1+\delta}-2\right)|U|$.
\end{lemma}

\begin{proof}
    Suppose otherwise, that is, there exists $U$ with $|U|\leq \log^C n$, but $|N_G(U)|< \left(\frac{d}{1+\delta}-2\right)|U|$.
    Then, we have $|U\cup N_G(U)|<\frac{(d-1-\delta)|U|}{1+\delta}<d|U|\le \log^{C+1} n=\log^{C'}n$.
    Thus, by \ref{prop2C},
    \begin{align*}
        e(U\cup N_G(U))\leq (1+\delta)|U\cup N_G(U)|<(d-1-\delta)|U|.
    \end{align*}
    On the other hand,
    \begin{align*}
        e(U\cup N_G(U))\ge \sum_{u\in U}d_G(u)-e(U)= d|U|-e(U)\geq (d-1-\delta)|U|,
    \end{align*}
    where the last inequality follows from \ref{prop2C}.
    Comparing the two bounds on $e(U\cup N_G(U))$ leads to a contradiction.
\end{proof}

While the assumptions in Theorems \ref{general} and \ref{constant} are rather similar, we will see in Lemma \ref{balls} why the stronger assumption \ref{prop2C} is necessary for this proof method to go through.

With Lemma \ref{basic2} in hand, we have the following argument, which we shall call \textit{the key argument} for the rest of this section, as it will accompany us throughout the entire proof. We formulate this argument for $G[V_p]$, noting that it applies to $G[V_{p_1}]$ as well.

Suppose we have a vertex $v\in V(G)$ and reveal its component in $G[V_p]$ via the BFS algorithm. Assume that $|C_{G[V_p]}(v)|=k$ for some $k\le \log^C n$. Set $S\coloneqq C_{G[V_p]}(v)$. By Lemma \ref{basic2}, during the BFS algorithm there were at least $|S|+|N_G(S)|\ge \left(\frac{d}{1+\delta}-1\right)k$ queried vertices, of which we have received at most $|S|=k$ positive answers. 
The probability of such an event to occur is bounded from above by the probability that $Bin\left(\left(\frac{d}{1+\delta}-1\right)k, \frac{\alpha}{d-1}\right)\le k$.

Using the key argument, we show that \textbf{whp}, $G[V_p]$ (as well as $G[V_{p_1}]$) contains no components of `intermediate' sizes (where here `intermediate' is of size between $\Omega_{\alpha}(\log n)$ and $\log^C n$). Formally, let $v\in V(G)$ and $k\le \log^C n$, and set $X\sim Bin\left(\left(\frac{d}{1+\delta}-1\right)k, \frac{\alpha}{d-1}\right)$. By the key argument, and using Claim \ref{cher}, we obtain that
\begin{align}\label{details_const}
    \mathbb{P}(|C_{G[V_p]}(v)|=k)\le \mathbb{P}\left(X\leq k\right)\le 2\exp\left(-\frac{(\mathbb{E}[X]-k)^2}{3\cdot \mathbb{E}[X]}\right).
\end{align}
Observe that
\begin{align}\label{details_const_2}
    \frac{1-\delta}{1+\delta}\cdot \alpha k\le \mathbb{E}[X]=\left(\frac{d}{1+\delta}-1\right)\cdot \frac{1}{d-1}\cdot \alpha k\le \alpha k.
\end{align}
Hence, using \eqref{details_const}, \eqref{details_const_2} and afterwards the union bound over all $v\in V(G)$ and $\frac{9\alpha}{\left(\frac{1-\delta}{1+\delta}\alpha-1\right)^2}\log n\le k\le \log^C n$, we have that \textbf{whp} $G_p$ contains no components of size between $\frac{9\alpha}{\left(\frac{1-\delta}{1+\delta}\alpha-1\right)^2}\log n$ and $\log^C n$.

Hence, as expected, we define
\begin{align*}
    W=\left\{v\in V(G)\;\bigg|\; |C_{G_1}(v)|=\frac{9\alpha}{\left(\frac{1-\delta}{1+\delta}\alpha-1\right)^2}\log n\right\}.
\end{align*} 

We are now turn to show that all `large' components in $G[V_{p_1}]$, $W$, merge after sprinkling with $p_2$ into one component in $G[V_p]$. By the key argument, and using again Lemma \ref{basic2} to show that the ball (in $G$) of radius $4\log_d\log n+1$ around each vertex $v\in V(G)$ contains many vertices, we show that \textbf{whp} the distance between any vertex $v\in V(G)$ from `large' components in $G[V_{p_1}]$, $W$, is at most $4\log_d\log n+1$ (these claims are analogous to Lemmas \ref{dist_2} and \ref{W_2}). Now we can apply essentially the same method as in Lemma \ref{one_comp_2} to show that \textbf{whp} there is a component $K$ in $G[V_p]$ such that $W\subseteq V(K)$. 

We also must check, as in the proof of Theorem~\ref{general} that sprinkling does not create new large components. More precisely, \textbf{whp}, there is no component in $G[V_p]$ of size at least $\log^C n$, which does not intersect $W$. This is done as in the proof of Lemma \ref{big_intersect_W}, and using again the key argument.

All left to do is to show that the unique `large' component which appeared in $G[V_p]$ is of the correct size. This argument goes quite differently compared to the corresponding argument in the proof of Theorem~\ref{general}, therefore, we provide here the proof in full.

Before that, we need one more lemma regarding the local structure of $G$, showing that there are ‘many’ vertices in $G$ that are far from any cycle. Formally,

\begin{lemma}\label{balls}
        There is a set $X \subseteq V(G)$ with $|X|\geq \left(1-\frac{1}{(d-1)^{\frac{1}{16\delta}}}\right)n$ such that for every $v \in X$, there are no cycles in $B\left(v,\frac{1}{16\delta}\right)$.
    \end{lemma}

This lemma appeared and proved in \cite{zbMATH08098276}, but in order to show the critical role of property \ref{prop2C}, we give here its complete proof.

\begin{proof}
     We first claim that in $G$, every two cycles of length at most $\frac{1}{4\delta}$ are at a distance of at least $\frac{1}{4\delta}$. Otherwise, let $U$ be the vertices of these two cycles, together with the vertices of a path of length at most $\frac{1}{4\delta}$ connecting them. Then, $|U|\leq 2\cdot\frac{1}{4\delta}+\frac{1}{4\delta}=\frac{3}{4\delta}$ and $e(U)\geq |U|+1$ (since $U$ is connected and contains at least two cycles). On the other hand, by \ref{prop2C}, 
     $$e(U)\leq (1+\delta)|U|\leq |U|+\delta\cdot \frac{3}{4\delta}<|U|+1.$$
     Let $m$ be the number of cycles of length at most $\frac{1}{4\delta}$ in $G$, denote them by $C_1, \dots , C_m$. For every $i \in [m]$, let $N_i$ be the set of vertices in $V(G)$ at distance at most $\frac{1}{8\delta}$ from $C_i$. We then have for every $i\neq j$, $N_i\cap N_j=\emptyset$ (as otherwise we would have two cycles of length at most $\frac{1}{4\delta}$ at distance at most $2\cdot\frac{1}{8\delta}=\frac{1}{4\delta}$). By the same reasoning, for every $i \in [m]$, $G[N_i]$ has only one cycle. We thus obtain $m\cdot (d-1)^{\frac{1}{8\delta}}\leq n$, that is, $m \leq \frac{n}{(d-1)^{\frac{1}{8\delta}}}$.

     Since $m\cdot (d-1)^{\frac{1}{16\delta}}\leq \frac{n}{(d-1)^{\frac{1}{16\delta}}}$, we conclude that there are at most $\frac{n}{(d-1)^{\frac{1}{16\delta}}}$ vertices in $G$ that are at distance at most $\frac{1}{16\delta}$ from a cycle of length at most $\frac{1}{4\delta}$. Hence, there is a set $X \subseteq V(G)$ with $|X|\geq \left(1-\frac{1}{(d-1)^{\frac{1}{16\delta}}}\right)n$ such that every $v \in X$ is at a distance of at least $\frac{1}{16\delta}$ from any cycle of length at most $\frac{1}{4\delta}$. Thus, every $v \in X$ satisfies that $B\left(v,\frac{1}{16\delta}\right)$ contains no cycles.
\end{proof}

We are now ready to show the last step in the proof of Theorem~\ref{constant}, which is showing that the largest component of $G[V_p]$, which we denote by $L_1$, is of size $(1+o_\delta(1))x\cdot n$.

The probability a vertex belongs to a component of size at least $\log^C n$ in $G[V_p]$ is stochastically dominated by the vertex version of Galton-Watson tree with offspring distribution $Bin(d,p)$, as expressed in Section \ref{intro}. Thus, by standard results (see, for example, \cite{zbMATH06976511}, Theorem 4.3.12), we yield \textbf{whp} that $|L_1|\le (1+o(1))x\cdot n$.

For the lower bound on the size of $L_1$, we use Lemma \ref{balls}.
Let $Z_1$ be the random variable counting the number of vertices in components of size at least $\frac{1}{16\delta}$ in $G_2$. By Lemma \ref{balls}, there exists a set $X \subseteq V(G)$ with $|X|\geq \left(1-\frac{1}{(d-1)^{\frac{1}{16\delta}}}\right)n$ such that for every $v \in X$, there are no cycles in $B\left(v,\frac{1}{16\delta}\right)$. Hence, by standard results, we have that for every $v\in X$, $\mathbb{P}(|C_{G_2}(v)|\geq \frac{1}{16\delta})\geq (1-o_\delta(1))x$, where $o_\delta(1)$ tends to zero as $\delta$ tends to zero. Thus $\mathbb{E}\left[|Z_1|\right]\geq (1-o_\delta(1)) x \left(1-\frac{1}{(d-1)^{\frac{1}{16\delta}}}\right)n=(1-o_\delta(1))x\cdot n$. To show that $|Z_1|$ is well concentrated around its mean, consider the vertex-exposure martingale. Every vertex can change the value of $|Z_1|$ by at most $\frac{d}{16\delta}$. Hence, by part 2 of Claim \ref{hoef}:
$$\mathbb{P}\bigg(\Big||Z_1|-\mathbb{E}[|Z_1|]\Big|\geq n^{2/3}\bigg)\leq 2\exp \left(-\frac{n^{4/3}}{2n}\cdot\frac{256\delta^2}{d^2}\right)=o(1),$$
where the last equality holds since $d=O(1)$.
Therefore, \textbf{whp}, $|Z_1|\geq (1-o_\delta(1))x\cdot n$.
Let $Z_2$ be the random variable counting the number of vertices in components of $G_2$ whose size lies in the interval $\left[\frac{1}{16\delta}, \log^C n\right]$. By the key argument, $\mathbb{E}[|Z_2|]\le o_\delta(1)x\cdot n$. Once again, let us consider the vertex-exposure martingale. Every vertex can change the value of $|Z_2|$ by at most $d\log^C n$. Hence, by part 2 of Claim \ref{hoef}:
$$\mathbb{P}\bigg(\Big||Z_2|-\mathbb{E}[|Z_2|]\Big|\geq n^{2/3}\bigg)\leq 2\exp \left(-\frac{n^{4/3}}{2n}\cdot\frac{1}{d^2\log^{2C} n}\right)=o(1),$$
where the last equality holds since $d=O(1)$. Therefore, we obtain that the number of vertices in components of size at least $\log^C n$ in $G_2$ is \textbf{whp} $(1-o_\delta(1))x\cdot n-(1+o(1))o_\delta(1)x\cdot n$.
Thus, since $\delta$ is sufficiently small, \textbf{whp} there are at least $(1-o_\delta(1))x\cdot n$ vertices in components of size at least $\log^C n$ in $G_2$. As argued above, \textbf{whp}, every component of size at least $\log^C n$ in $G[V_p]$ intersects with $W$. Thus, the number of vertices in components that intersect with $W$ is at least $(1-o_\delta(1))x\cdot n$. Moreover, all the vertices in $W$ merge into a unique component $L_1$, and hence \textbf{whp} $|L_1|\ge (1-o_\delta(1))x\cdot n$. Altogether, we have that \textbf{whp} $|L_1|=(1+o_\delta(1))x\cdot n$.

\section{Proof of Theorem~\ref{no_second}}\label{Pno_second}

The construction given here is similar in spirit to that given in \cite[Section 5.2]{diskin2024percolationisoperimetry}.

The family of $(n,d,\lambda)$-graphs features a fundamental result by Alon and Milman \cite{zbMATH03875335}. We give here a version of this result, which will be useful for us:

\begin{claim} \label{eml}
    Let $G = (V, E)$ be an $(n, d, \lambda)$-graph. Then for every $U\subseteq V(G)$,
    \begin{align*}
        e(U,U^c)\ge \frac{d-\lambda}{n}|U|(n-|U|).
    \end{align*}
\end{claim}

The construction and its expansion properties are described in Section \ref{construction}. Section \ref{noERCP} is devoted to proving the likely existence of at least two components of size $\Omega\left(d\log n\right)$ (which violates ERCP).

\subsection{The Construction}\label{construction}

Let $\alpha>0$ be a constant and let $\epsilon>0$ be a sufficiently small constant. Let $b\le \frac{1}{500\epsilon}$. Set $c\coloneqq c(\alpha)>0$ to be sufficiently small. Let $n\in \mathbb{N}$ be sufficiently large and let $d\coloneqq d(n)\in \mathbb{N}$ such that $d$ is even, $d<n^{1/5}$, $d=\omega(1)$. Define $p=\frac{1+\epsilon}{d}$ and $k=d^2\log n$. In addition, assume that $b\ge 2$. (This is possible because if a graph satisfies \ref{prop0B} for $b\ge 2$, then it also satisfies \ref{prop0B} for any constant smaller than $2$.)

Let $H_1$ be a random $10b$-regular graph on $n$ vertices. Note that due to \cite{zbMATH05315244}, $H_1$ is \textbf{whp} an $\left(n,10b,3\sqrt{10b-1}\right)$-graph. Since $n/k=\omega(1)$ and $b$ is a constant, by \cite{zbMATH03344609}, $H_1$ admits an equitable (proper) coloring of its vertices with $n/k$ colors, i.e., we can partition $V(H_1)$ into $V(H_1)=\bigcup_{i=1}^{n/k}A_i$ so that $A_i$ is an independent set of $k$ vertices in $H_1$, for every $1\le i\le n/k$.

Let $H_2$ be a random $(d-10b)$-regular graph on $k$ vertices. Then $H_2$ is \textbf{whp} a $(k,d-10b,\tau)$-graph with $\tau=O\left(\sqrt{d-10b}\right)$ (see
\cite{zbMATH05037081,zbMATH07751065,zbMATH07036340} and the references therein for even stronger results).

We use $H_1$ and $H_2$ to construct our desired graph $G$. Consider $H_1$ and a partition of $V(H_1)$ into $\bigcup_{i=1}^{n/k}A_i$ as described above. In each $\{A_i\}_{i=1}^{\frac{n}{k}}$, embed a copy of $H_2$ in an arbitrary way. Notice that $G$ is a $d$-regular graph on $n$ vertices. Let us verify that $G$ meets the expansion properties detailed in Theorem~\ref{no_second} (Properties \ref{prop0B}-\ref{prop2B}).

\begin{claim}\label{first}
    The graph $G$ satisfies \ref{prop0B}.
\end{claim}

\begin{proof}[Proof of Claim \ref{first}]
    Observe that $H_1$ is a spanning subgraph of $G$, and hence if $H_1$ fulfills \ref{prop0B}, then so does $G$. By Claim \ref{eml}, for every $U\subseteq V(G)$ with $|U|\le \frac{n}{2}$, 
    \begin{align*}
        e_G(U,U^c)&\ge e_{H_1}(U,U^c)\ge \frac{10 b-3\sqrt{10 b-1}}{n}|U|(n-|U|)\\
        &\ge \left(5 b-\frac{3}{2}\sqrt{10 b-1}\right)|U|\ge b|U|.\qedhere
    \end{align*}
\end{proof}

\begin{claim} \label{second}
    The graph $G$ satisfies \ref{prop1B}.
\end{claim}

\begin{proof}[Proof of Claim \ref{second}]
    Let $U\subseteq V(G)$ be a set of vertices of size $|U|\le c\cdot d\log n$. For every $1\le i\le \frac{n}{k}$, define $U_i\coloneqq U\cap A_i$.
    Since $|U|\le c\cdot d\log n$, we have $|U_i|\le c\cdot d\log n$, for every $1\le i\le \frac{n}{k}$. 
    
    We use the fact that random $d$-regular graphs exhibit typically almost perfect vertex expansion, with expansion factor of $(1-\beta)d$, for sets of size $o\left(\frac{n}{d}\right)$ and for small constant $\beta>0$ (see, for example, Lemma 6 in \cite{zbMATH06305468}).
    By applying this fact to $G[A_i]$, which is a copy of the random $(d-10b)$-regular graph $H_2$, we obtain that $|N_{G[A_i]}(U_i)|\ge \left(1-\frac{\alpha}{2}\right)(d-10b)|U_i|\ge (1-\alpha)d|U_i|$, where the last inequality holds since $d=\omega(1)$.
    Therefore, 
    \begin{align*}
        |N_G(U)|\ge \sum_{i=1}^{n/k} |N_{G[A_i]}(U_i)|\ge (1-\alpha)d\cdot \sum_{i=1}^{n/k} |U_i| \ge (1-\alpha)d|U|.&\qedhere
    \end{align*} 
\end{proof}

Finally,

\begin{claim} \label{third}
    The graph $G$ satisfies \ref{prop2B}.
\end{claim}

\begin{proof}[Proof of Claim \ref{third}]
    Let $U\subseteq V(G)$ be a set of vertices of size $|U|\le \frac{\alpha}{2}\cdot d^2\log n$. For every $1\le i\le \frac{n}{k}$, define $U_i\coloneqq U\cap A_i$, and notice that 
    \begin{align*}
        e_G(U,U^c)\ge \sum_{i=1}^{n/k} e_{G[A_i]}(U_i,A_i\backslash U_i).
    \end{align*}
    Now, as $|U|\le \frac{\alpha}{2}\cdot d^2\log n$, we have $|U_i|\le \frac{\alpha}{2}\cdot d^2\log n$, for every $1\le i\le \frac{n}{k}$. Hence, by Claim \ref{eml}, 
    \begin{align*}
        e_{G[A_i]}(U_i,A_i\backslash U_i)\ge \frac{(d-10b)-\tau}{d^2\log n}|U_i||A_i\backslash U_i|\ge \left(1-\frac{\alpha}{2}\right)(d-10b-\tau)|U_i|\ge (1-\alpha)d|U_i|,
    \end{align*}
    where the last inequality holds since $\tau=O\left(\sqrt{d-10b}\right)$ and $d=\omega(1)$.
    Therefore,
    \begin{align*}
        e_G(U,U^c)\ge \sum_{i=1}^{n/k} e_{G[A_i]}(U_i,A_i\backslash U_i) \ge \sum_{i=1}^{n/k}(1-\alpha)d|U_i|=(1-\alpha)d|U|.&\qedhere
    \end{align*}
\end{proof}

\subsection{The absence of ERCP}\label{noERCP}

Lastly, we show that ERCP fails for the graph $G$ we constructed.

\begin{claim}\label{PnoERCP}
    \textbf{Whp}, there are at least two components of $G[V_p]$ of size $\Omega(d\log n)$.
\end{claim}

\begin{proof}
    We start by constructing an auxiliary graph $\Gamma$ whose vertex set is $\{A_i\}_{i=1}^{\frac{n}{k}}$, and $\left(A_i,A_j\right)\in E(\Gamma)$ ($i\neq j$) if their distance in $G$ is at most two, that is, if there exists a path in $G$ of length (in edges) at most $2$ that connects $A_i$ to $A_j$. Notice that the maximal degree of $\Gamma$, $\Delta(\Gamma)$, is at most $k\cdot (10b)^2$, and therefore $\Gamma$ contains an independent set, $I$, of size at least $\frac{|V(\Gamma)|}{k\cdot (10b)^2+1}\ge \frac{n}{200b^2k^2}$.
    
    Denote by $I_0=\{1\le i\le \frac{n}{k}\;|\; A_i\in I\}$. For every $i\in I_0$, we denote by $X_i$ the indicator of the event for which $G[V_p]$ contains a component $S\subseteq A_i$, of size $|S|\in [\epsilon\cdot d\log n,3\epsilon\cdot d\log n]$. Notice that since $I$ is an independent set of $\Gamma$, the random variables $\{X_i\}_{i\in I_0}$ are i.i.d.

    Let $i\in I_0$. Notice that $X_i=1$ if and only if $G[A_i\cap V_p]$ admits a connected component of size in the interval $[\epsilon\cdot d\log n,3\epsilon\cdot d\log n]$, and its neighbors outside of $A_i$ are not in $V_p$. We bound the probability that $X_i=1$ by exposing the vertices of $G[V_p]$ in two steps. We first reveal $A_i$, apply Theorem 2 from \cite{zbMATH07751060} to $G[A_i]\cong H_2$, and obtain, with probability at least $1-o(1)$, a connected component $S$ in $G[A_i\cap V_p]$ of size in the interval $[\epsilon\cdot d\log n,3\epsilon\cdot d\log n]$. In order to establish that $S$ is a connected component in the whole $G[V_p]$, the neighborhood of $S$ outside of $A_i$ must fall outside of the random subset $V_p$. We now expose $V(G)\backslash A_i$ (in particular, we expose the neighborhood $N_{G[V(G)\backslash A_i]}(S)$). Hence,
    
    \begin{align*}
        \mathbb{P}(X_i=1)&\ge (1-o(1))(1-p)^{\left|N_{G[V(G)\backslash A_i]}(S)\right|}\ge (1-o(1))(1-p)^{10b\cdot 3\epsilon\cdot d\log n}\\
        &\ge (1-o(1))\exp\left(-\left(\frac{1+\epsilon}{d}+\left(\frac{1+\epsilon}{d}\right)^2\right)10b\cdot 3\epsilon \cdot d\log n\right)\\
        &\ge \exp\left(-50b\epsilon\cdot \log n\right)\ge n^{-50b\epsilon}\ge n^{-0.1},
    \end{align*}
    where the second inequality yields by the fact that
    \begin{align*}
        |N_{G[V(G)\backslash A_i]}(S)|=|N_{H_1}(S)|\le 10b|S|\le 10b\cdot 3\epsilon \cdot d\log n,
    \end{align*}
    the third inequality holds since $1-x\ge \exp\left(-x-x^2\right)$ for $0<x<1$, and the last inequality holds because of the assumption $b< \frac{1}{500\epsilon}$.
    
    Let $X=\sum_{i\in I_0}X_i$. Then,
    
    \begin{align*}
        \mathbb{E}[X]\ge |I_0|\cdot n^{-0.1}\ge \frac{n}{200b^2k^2}\cdot n^{-0.1}=\frac{n^{0.9}}{200b^2d^4\log^2 n}.
    \end{align*}
    
    Now, since $\{X_i\}_{i\in I_0}$ are i.i.d. random variables, we find that \textbf{whp} $G[V_p]$ contains $\Omega\left(\frac{n^{0.9}}{d^4\log^2 n}\right)$ connected components $C$ such that $C\in [\epsilon\cdot d\log n,3\epsilon\cdot d\log n]$. In particular, since $d<n^{1/5}$, \textbf{whp} $X\ge 2$, as desired.
\end{proof}

\section{Concluding Remarks}\label{discussion}

We have proven that in the site percolation model on $d$-regular graphs, for both constant-degree and growing-degree cases, we obtain the ERCP under a very mild assumption on the global edge expansion property, supported by a good control over the expansion of sets of size up to $\text{poly}(d,\log n)$. 

Our result recovers and improves upon the previously known results for random $d$-regular graphs and for the hypercube $Q^d$, and thereby, resolves the questions concerning them in full. To cover pseudo-random graphs as well, we gave a concrete statement which shows that the ERCP holds in full for $(n,d,\lambda)$-graphs, under the standard assumption that the spectral ratio $\lambda/d$ is constantly small. In fact, our result provides a full answer to the open problem for $(n,d,\lambda)$-graphs. Another perhaps surprising fact revealed in this research is a fundamental difference between site- and bond-percolation --- we showed that under the assumptions that ensure the bond-percolation analogous ERCP according to \cite{diskin2025componentslargesmalli}, and even under stronger assumptions, the ERCP is not guaranteed in site percolation setting.

We have also discussed the limitations of the proof of Theorem~\ref{general} (we have to assume that $d<n^{\gamma}$ for some $\gamma>0$) and the tightness of Theorem~\ref{general} (in Theorem~\ref{no_second}) --- we provide a construction demonstrating that Theorem~\ref{general} is tight in some sense. Nevertheless, there is still a gap between the assumptions in the two theorems. We tend to believe that the assumption $d<n^{\gamma}$ is just a technical artifact of the proof and that the correct dependence on $d$ and $\log n$ is as stated in the following conjecture:

\begin{conjecture}
    For every constants $0<\alpha<1$, $c_1,c_2>0$, and for every sufficiently small constant $\epsilon>0$, there exist $C_1\coloneqq C_1(\epsilon, \alpha, c_1,c_2)>0$ and $C_2\coloneqq C_2(\epsilon, \alpha, c_1,c_2)>0$ such that the following holds. Let $n,d\coloneqq d(n)\in \mathbb{N}$ be sufficiently large integers such that $d=o(n)$. Let $G=(V,E)$ be a $d$-regular graph on $n$ vertices satisfying:
    \begin{enumerate}[\arabic*{}.]
        \item For every $U\subseteq V(G)$ with $|U|\leq \frac{n}{2}$, $e_G(U,U^c)\geq c_1|U|$;
        \item For every $U\subseteq V(G)$ with $|U|\leq C_1\log n$, $|N_G(U)|\ge c_2d|U|$;
        \item For every $U\subseteq V(G)$ with $|U|\leq C_2d^2\log n$, $e_G(U,U^c)\ge \left(1-\alpha\right)d|U|$.
    \end{enumerate}
    Then, $G$ satisfies the ERCP.
\end{conjecture}

\bibliographystyle{abbrv}
\bibliography{Bib}

\end{document}